\documentclass[11pt,english]{article}

\usepackage{amsfonts,babel,amssymb,amsmath,latexsym}

\usepackage{graphicx}

\newcommand\qan{{\quad\hbox{and}\quad}}
\newcommand\qin{{\quad\hbox{in}\quad}}
\newcommand\qon{{\quad\hbox{on}\quad}}
\newcommand\disp{\displaystyle}
\newcommand\bsi{{\boldsymbol{\sigma}}}
\newcommand\bta{{\boldsymbol{\tau}}}
\newcommand\bze{{\boldsymbol{\zeta}}}
\newcommand\bnu{{\boldsymbol{\nu}}}
\newcommand\bch{{\boldsymbol{\chi}}}
\newcommand\bet{{\boldsymbol{\eta}}}
\newcommand\bva{{\boldsymbol{\varphi}}}
\newcommand\bps{{\boldsymbol{\psi}}}
\newcommand\bla{{\boldsymbol{\lambda}}}
\newcommand\brh{{\boldsymbol{\rho}}}
\newcommand\bxi{{\boldsymbol{\xi}}}

\newcommand\f{{\mathbf f}}
\newcommand\g{{\mathbf g}}
\renewcommand\L{{\mathbf L}}
\renewcommand\H{{\mathbf H}}
\newcommand\F{{\mathbf F}}
\newcommand\G{{\mathbf G}}
\newcommand\X{{\mathbf X}}
\newcommand\Y{{\mathbf Y}}
\newcommand\I{{\mathbf I}}
\newcommand\R{{\mathrm R}}
\renewcommand\S{\mathbf{S}}
\newcommand\D{\mathbf{D}}
\newcommand\V{\mathbf{V}}
\newcommand\K{\mathbf{K}}
\newcommand\W{\mathbf{W}}
\newcommand\tr{{\mathrm{tr}\,}}
\newcommand\0{{\boldsymbol{0}}}
\newcommand\bdiv{\mathbf{div}\,}
\renewcommand\div{\mathrm{div}\,}
\renewcommand\u{\mathbf{u}}
\newcommand\bv{\mathbf{v}}
\newcommand\e{\mathbf{e}}
\newcommand\x{\mathbf{x}}
\newcommand\s{\mathbf{s}}
\newcommand\y{\mathbf{y}}
\newcommand\z{\mathbf{z}}
\newcommand\w{\mathbf{w}}
\renewcommand\a{\mathbf{a}}
\renewcommand\b{\mathbf{b}}

\newtheorem{lemma}{\sc Lemma}[section]
\newtheorem{theorem}{\sc Theorem}[section]
\newenvironment{proof}{\noindent{\it Proof.}}{\hfill$\square$}

\numberwithin{figure}{section}
\numberwithin{table}{section}
\numberwithin{equation}{section}

\title{New developments on the coupling of mixed-FEM and BEM\\
for the three-dimensional exterior Stokes problem}

\author{{\sc Gabriel N. Gatica}\thanks{CI$^2$MA and Departamento de Ingenier\'\i a Matem\'atica,
Universidad de Concepci\'on, Casilla 160-C, Concepci\'on, Chile,
email: {\tt ggatica@ci2ma.udec.cl}} \quad 
{\sc George C. Hsiao}\thanks{Department of Mathematical Sciences, University of Delaware,
Newark, DE 19716-2553, USA, e-mail: {\tt hsiao@math.udel.edu}} \\[0.2em]
{\sc Salim Meddahi}\thanks{Departamento de Matem\'aticas, Facultad de Ciencias,
Universidad de Oviedo, Calvo Sotelo s/n, Oviedo, Espa\~na,
e-mail: {\tt salim@uniovi.es}}\quad
{\sc Francisco-Javier Sayas}\thanks{Department of Mathematical Sciences, University of Delaware,
Newark, DE 19716-2553, USA, e-mail: {\tt fjsayas@math.udel.edu}. Partially funded by NSF grant DMS-1216356. } }

\date{ }

\begin{document}

\maketitle

\begin{abstract}
\noindent 
In this paper we consider the three dimensional exterior Stokes problem and study the solvability of the
corresponding continuous and discrete formulations that arise from the coupling of 
a dual-mixed variational formulation (in which the velocity, the pressure and the
stress are the original main unknowns) with the boundary integral equation method.
More precisely, after employing the incompressibility condition to eliminate the pressure, we consider the
resulting velocity-stress-vorticity approach with different kind of boundary conditions
on an annular bounded domain,  and couple the underlying equations with either one
or two boundary integral equations arising from the application of the usual and normal traces
to the Green representation formula in the exterior  unbounded region. As a result, 
we obtain saddle point operator equations, which are then  analyzed by the well-known
Babu\v ska-Brezzi theory. We prove the well-posedness of the  continuous
formulations, identifying previously the space of solutions of the associated 
homogeneous problem, and specify explicit hypotheses to be satisfied by the finite element and 
boundary element subspaces in order to guarantee the stability of the respective Galerkin schemes. 
In particular, following a similar analysis given recently for the
Laplacian, we are able to extend the classical Johnson \& N\'ed\'elec procedure
to the present case, without assuming any restrictive smoothness requirement  on
the coupling boundary, but only Lipschitz-continuity. In addition, and differently
from known approaches for the elasticity problem, we are also able to extend the
Costabel \& Han coupling procedure to the 3D Stokes problem
by providing a direct proof of the required coerciveness property, that is without
argueing by contradiction, and by using the natural norm of each space instead of 
mesh-dependent norms. Finally, we briefly describe concrete examples of discrete spaces satisfying 
the aforementioned hypotheses.



\end{abstract}

\section{Introduction}\label{section1}
The classical approach combining finite element (FEM) with boundary element methods (BEM) for solving
exterior boundary value problems in continuum mechanics, usually known as the coupling
of FEM and BEM, has been extensively employed since its creation during the second half of the seventies 
up to nowadays. The usual procedure is as follows. The underlying domain is first divided into two subregions 
by introducing an auxiliary boundary $\Gamma$, if necessary, so that the original exterior problem can be reformulated 
as a transmission problem through $\Gamma$. Next, the latter is reduced to an equivalent problem in the 
bounded inner region by imposing nonlocal boundary conditions on $\Gamma$ that are derived by employing 
boundary integral equation methods in the unbounded outer domain. The resulting nonlocal boundary value problem 
is then solved by a conventional Galerkin method, in which the boundary integral operators involved are discretized using finite 
element spaces on $\Gamma$. 

\medskip
While detailed surveys on most of the different ways of coupling BEM and FEM can be seen in 
\cite{Hsiao} and \cite[Chapter I]{GaticaHsiao}, we simply recall here that the most popular ones correspond to
the {\sl Johnson \&  N\'ed\'elec} (J \& N) and {\sl Costabel \& Han} (C \& H) 
procedures (cf. \cite{bj1979}, \cite{BrezziJohnsonNedelec}, \cite{c1987}, \cite{h1990}, \cite{jn1980}, 
and \cite{Zienkiewicz}), which employ the Green representation of the solution in the 
unbounded region. The success of the J \& N method, being
based on a single boundary integral equation on $\Gamma$ and the Fredholm theory, hinged on the fact
that certain boundary integral operators are compact, which usually requires $\Gamma$ to be smooth enough.
According to it, it was not possible, at least from a theoretical point of view, to employ this approach
when the coupling boundary was non--smooth, say for instance polygonal, which left out 
the possibility of utilizing classical finite element discretizations. Moreover,  the J \& N idea
seemed to be applicable only to the Laplace operator since for other elliptic systems, such as the
elasticity one, and irrespective of the smoothness of the boundaries, the aforementioned compactness
did not hold. One attempt to overcome this was suggested in \cite{BielakMacCamy} where the 
underlying transmission problem was replaced by one employing the pseudostress instead of the 
usual stress. As a consequence, the foregoing mapping property was achieved, but the coupling boundary 
was still required to be smooth enough. One has to admit, however, that the above described drawbacks were
mainly theoretical since no failure of the corresponding discrete schemes was ever reported by users of 
the method in problems where those hypotheses were not met. Any way, in order to circumvent these
apparent difficulties, suitable modifications of the original  J \& N method, 
in which neither the compactness nor the smoothness play any role, were proposed 
by {\sl Costabel} and {\sl Han} in \cite{c1987} and \cite{h1990}, respectively. 
Both techniques are based on the addition of a boundary integral equation for the normal derivative 
(resp. traction in the case of elasticity). The former leads to a symmetric and non-positive definite scheme, 
while the latter, on the contrary, yields a positive definite and non-symmetric scheme. Nevertheless, 
and since the only difference between these formulations lies on the sign of an integral identity, 
from now on we simply refer to either one of them as the C \& H approach. Further and later contributions 
in this direction, including applications to nonlinear problems and coupling with mixed-FEM, non--conforming FEM, 
local discontinuous Galerkin, and hybridizable discontinuous Galerkin methods, can be found in 
\cite{affkmp-CM-2013}, \cite{bcs1996}, \cite{bgs-IMANUM-2008}, \cite{CarstensenFunken2}, 
\cite{CarstensenFunken1}, \cite{cfs-1997}, \cite{CockburnSayas},
\cite{cs-1990}, \cite{gh-SINUM-2000}, \cite{ghs-MC-2010}, \cite{gh-ZAA-1989}, 
\cite{gs-MC-2006}, \cite{gw1996}, \cite{gw1997}, \cite{Meddahi}, and the references therein.

\medskip
The whole picture on the coupling of FEM and BEM, and particularly the widely accepted fact
since the eighties concerning the lack of further applicability and usefulness of the J \& N method, changed dramatically 
with \cite{s-SINUM-2009}. More precisely, it was proved in this paper, without any need of applying Fredholm theory nor 
assuming smooth domains, that all Galerkin methods for this approach are actually stable, thus 
allowing the coupling boundary $\Gamma$ to be polygonal/polyhedral. As a consequence, 
the classical J \& N method was begun to be considered 
as a real competitor of the C \& H approach. In other words, the appearing of \cite{s-SINUM-2009} gave
rise to several new contributions within this and related topics. Indeed, we first refer to 
\cite{mss-MC-2011} where the corresponding extension to the combination of mixed-FEM and BEM 
on any Lipschitz-continuous interface $\Gamma$ was successfully developed. Furthermore, the analysis of the 
quasi--symmetric procedure from \cite{BielakMacCamy} was improved in \cite{ghs-NM-2012} by 
showing that the interface $\Gamma$ can also be taken polygonal/polyhedral,
and that in the case of the elasticity problem, the coupling can be performed by employing the usual stress
instead of the pseudostress. In addition, a new and extremely simplified proof of the main result in \cite{s-SINUM-2009},
by showing directly ellipticity of the operator equation, was provided in \cite{ghs-NM-2012}.
An alternative proof of this ellipticity result has been recently given in \cite{s-SINUM-2011}, using a
particular expression of the Steklov--Poincar\'{e} operator, which is based on a Schur complement 
of a perturbation of the Calder\'{o}n projector. Some comments on the consequences of the new theory of non-symmetric coupling of BEM and FEM can be found in the
republishing  of \cite{s-SINUM-2009} as \cite{s-SR-2013}. Nevertheless, the utilization of mixed-FEM 
instead of the usual FEM, and the application of the J \& N coupling procedure 
to the 3D Stokes and similar elliptic systems such as Lam\'e, is still missing. 
Moreover, most of the related works available in the literature involve either
2D problems or just the coupling of BEM and the usual FEM (see, e.g. \cite{g-SINUM-1987}, 
\cite{r-EABE-2011}, \cite{r-IJNMF-2012}, and \cite{s-MMAS-1983}). In addition,  the analysis of 
the C \& H approach for the coupling of mixed-FEM and BEM has not yielded too satisfactory 
results when it has been applied to the elasticity problem (see, e.g. \cite{bcs1996}).

\medskip
According to the above bibliographic discussion, and specially motivated by the recent results
from \cite{s-SINUM-2009}, \cite{mss-MC-2011}, and \cite{ghs-NM-2012}, we now aim to analyze the coupling 
of mixed-FEM and BEM, as applied to the 3D exterior Stokes problem, by utilizing both the J \& N
and the C \& H approaches. More precisely, we extend the first method to the present case, without 
assuming any smoothness requirement on the interface, but only Lipschitz-continuity. 
Furthermore, and differently from the analysis in \cite{bcs1996} for the elasticity problem, 
we are also able to extend the second coupling procedure to the 3D Stokes problem
by providing a direct proof of the required coerciveness property, that is without
argueing by contradiction, and by using the natural norm of each space 
instead of mesh-dependent norms. Our results here can be easily extended to the
2D and 3D Lam\'e systems.

\medskip
The rest of this paper is organized as follows. In Section \ref{section2} we introduce the exterior boundary value problem
of interest by describing it as the transmission problem between the non-homogeneous Stokes equation
(holding in a bounded  annular domain $\Omega^-$) and the homogeneous Stokes equation (holding in an 
unbounded exterior region $\Omega^+$). The incompressibility condition is employed here to eliminate
the pressure so that the stress tensor and the velocity vector become the main unknowns of the resulting transformed 
problem. The dual-mixed formulations in $\Omega^-$ for different boundary conditions on the interior boundary 
of this region are derived in Section \ref{section3}. Then, in Section \ref{section4} we recall the main aspects
and properties of the boundary integral equation approach as applied to the homogeneous Stokes equation
in $\Omega^+$. Next, in Section \ref{section5} we derive and analyze the coupled variational formulations that
arise from the combination of the dual-mixed approach in $\Omega^-$ with the boundary integral equation method
in $\Omega^+$. We first identify the solutions of the associated homogeneous problems and then establish the well-posedness
of the continuous formulations. In particular, the classical Johnson \& Nedelec procedure, which employs a single
boundary integral equation and yields a non-symmetric scheme, is extended to the present case by requiring only 
a Lipschitz-continuous coupling boundary. In addition, the Costabel \& Han approach, which makes use of two 
bounday integral equations and leads to a symmetric formulation, is also successfully analyzed with the natural
norms of the spaces involved and through direct proofs of the required continuous and discrete coeciveness properties.
Finally, in Section \ref{section6} we consider the Galerkin schemes arising from the coupled formulations studied 
in Section \ref{section5}, and provide explicit hypotheses to be satisfied by the respective discrete spaces in order to 
guarantee their corresponding solvability and stability. Moreover, concrete examples of finite element and boundary element subspaces verifying those conditions are also identified here.

\medskip
We end this section with some notations to be used below. Given any Hilbert space $U$,
we denote by $U^3$ and $U^{3\times 3}$, respectively, the space of vectors and square matrices
of order $3$ with entries in $U$. In particular, the identity matrix of
$\R^{3\times 3}$ is $\I$, and given $\bta:=(\tau_{ij})$, $\bze:=(\zeta_{ij})\in\R^{3\times 3}$,
we write as usual
\[
\bta^{\tt t}\,:=\,(\tau_{ji})\,,\quad \tr\bta\,:=\,\sum_{i=1}^3
\tau_{ii}\,,\quad \bta^{\tt d}:=\bta-\frac{1}{3}\,\tr(\bta)\,\I\,,\quad\hbox{and}\quad
\bta : \bze\,:=\,\sum_{i,j=1}^3 \tau_{ij}\,\zeta_{ij}\,.
\]
Also, in what follows we utilize the standard terminology for Sobolev spaces and norms. However, 
given a domain $\mathcal O$, a closed Lipschitz curve $\Sigma$, and $r\in \R$, we 
simplify notations and define
\[
\mathbf H^r(\mathcal O)\,:=\, [H^r(\mathcal O)]^3\,,\quad
\mathbb H^r(\mathcal O)\,:=\,[H^r(\mathcal O)]^{3\times 3}\,,
\qan \mathbf H^r(\Sigma)\,:=\, [H^{r}(\Gamma)]^3\,.
\]
In the special case $r=0$ we usually write $\mathbf L^2(\mathcal O)$,
$\mathbb L^2(\mathcal O)$, and $\mathbf L^2(\Gamma)$ instead of $\mathbf H^0(\mathcal O)$, 
$\mathbb H^0(\mathcal O)$, and $\mathbf H^0(\Gamma)$, respectively.
The corresponding norms are denoted by $\|\cdot\|_{r,\mathcal O}$ 
(for $H^r(\mathcal O)$, $\mathbf H^r(\mathcal O)$, and $\mathbb H^r(\mathcal O)$) 
and $\|\cdot\|_{r,\Gamma}$ (for $H^r(\Gamma)$ and $\mathbf H^r(\Gamma)$).
In addition, denoting by $\bdiv$ the usual divergence operator $\div$ acting on the rows
of a tensor, we define the Hilbert space
\[
\mathbb H(\bdiv;\mathcal O)\,:=\,\big\{ \bta \in \mathbb L^2(\mathcal O):\quad 
\bdiv \bta \in \L^2(\mathcal O)\big\}\,,
\]
and the subspace
\begin{equation}\label{tilde-hbdiv}
\widetilde{\mathbb H}(\bdiv;\mathcal O)\,:=\,\left\{\,\boldsymbol\tau\,\in\,\mathbb H(\bdiv;\mathcal O):
\quad \int_{\mathcal O} \tr \boldsymbol\tau \,=\, 0\,\right\}
\end{equation}
which are both endowed with the norm
\[
\|\bta\|_{\bdiv;\mathcal O}\,:=\,\Big\{\,\|\bta\|^2_{0,\mathcal O}\,+\,\|\bdiv \bta\|^2_{0,\mathcal O}\,\Big\}^{1/2}
\qquad\forall\,\bta\,\in\,\mathbb H(\bdiv;\mathcal O)\,.
\]
Note that there holds the decomposition:
\begin{equation}\label{decomp-hdiv-1}
\mathbb H(\bdiv;\mathcal O)\,=\,\widetilde{\mathbb H}(\bdiv;\mathcal O)\,\oplus\,P_0(\mathcal O)\,\I\,,
\end{equation}
where $P_0(\mathcal O)$ is the space of constant polynomials on $\mathcal O$. 

\medskip
Finally, throughout the paper we employ $\0$  to denote a generic
null vector, and use $C$ and $c$, with or without subscripts, bars, tildes or hats, to denote
generic constants independent of the discretization parameters, which may take
different values at different places.
\section{The boundary value problem}\label{section2}
Let $\Omega_0$ be a bounded Lipschizt-continuous domain in $\R^3$ with boundary
$\Gamma_0$, let $\Omega^-$ be the annular region bounded by $\Gamma_0$ and
another Lipschitz-continuous surface $\Gamma$ whose interior contains $\bar\Omega_0$,
and let $\Omega^+\,:=\,\R^3 \backslash \big(\bar\Omega_0 \cup \bar \Omega^-\big)$
(see Figure 2.1 below). We consider a steady incompressible flow in the region $\R^3\backslash\Omega_0$, 
under the action of external forces on $\bar\Omega^-$, and are interested
in determining the velocity, the pressure, and the stress of the corresponding
fluid. More precisely, given $\f\in \L^2(\Omega^-)$, 
we seek a vector field $\u$, a scalar field $p$, and a tensor field $\bsi$
such that:
\begin{equation}\label{transmission-problem}
\begin{array}{c}
\disp
\bsi\,=\,\Xi[\u,p]\,:=\,2\,\mu\,\e(\u)\,-\,p\,\I \qan \div \u\,=\,0 \qin \Omega^-\,\cup\,\Omega^+\,,\\[2ex]
\disp
\bdiv\bsi \,=\,-\,\f\qin \Omega^-\,, \quad \bdiv\bsi \,=\,\0 \qin \Omega^+\,, \quad \hbox{\sc BC} \qon \Gamma_0\,, \\[2ex]
\disp
[\u]\,:=\,\u^-\,-\,\u^+\,=\,\0 \qan [\bsi \bnu ]\,:=\,(\bsi \bnu )^-\,-\,(\bsi \bnu )^+\,=\,\0 \qon\Gamma\,,\\[2ex]
\disp
\u(\x)\,=\,O(\|\x\|^{-1}) \qan p(\x)\,=\,O(\|\x\|^{-2})\quad\hbox{as}\quad \|\x\|\to +\infty\,,
\end{array}
\end{equation}
where $\hbox{\sc BC}$ stands for a suitable boundary condition on $\Gamma_0$, which will
be specified later on. Hereafter, $\Xi$ is the stress operator acting on the velocity/pressure pair, 
$\mu$ is the kinematic viscosity of the fluid, $\e(\u)\,:=\,\frac1{2}\,\big(\nabla \u + (\nabla\u)^t\big)$
is the strain tensor (or symmetric part of the velocity gradient), $\bnu$ is the unit normal on $\Gamma_0$ 
and $\Gamma$ pointing inside $\Omega^-$ and $\Omega^+$, respectively, 
\[
\u^\pm(\x)\,:=\,\lim_{\stackrel{\tilde\x\,\to\,\x}{\tilde\x\,\in\,\Omega^\pm}}\,\u(\tilde\x)\qquad
\forall\,\x\,\in\,\Gamma\,,
\]
and
\[
(\bsi \bnu )^\pm(\x) \,:=\,\lim_{\stackrel{\tilde\x\,\to\,\x}{\tilde\x\,\in\,\Omega^\pm}}\,
\bsi(\tilde\x)\,\bnu(\x)\qquad\forall\,\x\,\in\,\Gamma\,.
\]

\medskip

\begin{center}

\includegraphics[height=6cm]{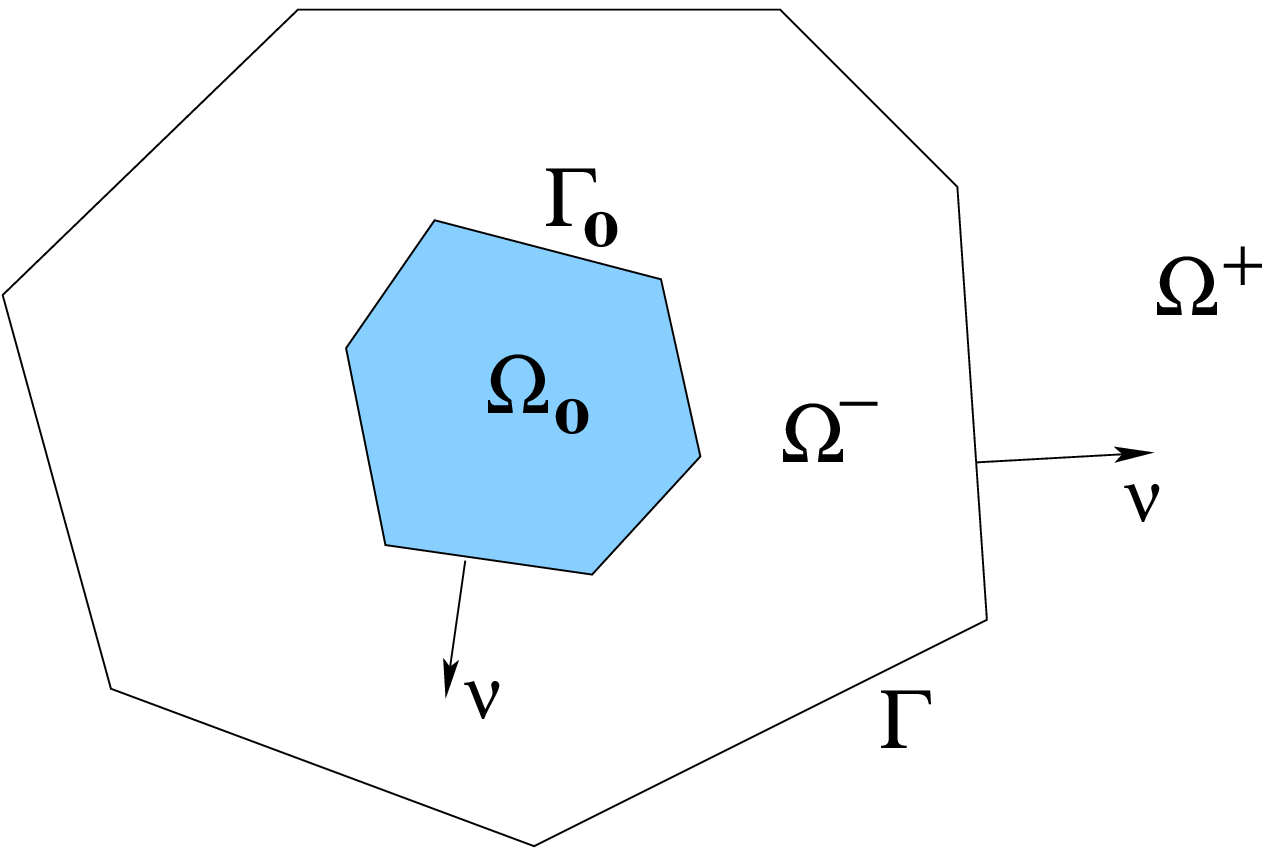}

\bigskip

{\bf Figure 2.1}: Geometry of the problem.

\end{center}

\medskip\noindent
Note that, thanks to the incompressibility condition given by
$\div\,\u\,=\,0 \qin \Omega^-\,\cup\,\Omega^+$, there holds
$\bdiv\,\bsi\,=\,\mu\,\Delta\,\u\,-\,\nabla\,p\qin \Omega^-\,\cup\,\Omega^+$, 
which means that the second row of \eqref{transmission-problem} becomes the
nonhomogeneous and homogeneous Stokes equations in $\Omega^-$ and $\Omega^+$,
respectively. However, since we are going to apply a mixed variational formulation
in $\Omega^-$ and the associated boundary integral equation approach in $\Omega^+$,
we need to keep $\bsi$ as an independent unknown.

\medskip
Throughout the rest of the paper, and without loss of generality, we assume that
$\mu = 1/2$. Otherwise, we just redefine $p$ as $p/2\mu$ and
let $\bsi\,=\,\Xi[\u,p]\,:=\,\e(\u)\,-\,p\,\I \qin \Omega^-\,\cup\,\Omega^+$,
which yields the datum $\f$ to be replaced by $\f/2\mu$. In addition, it is
easy to see, using that $\tr\e(\u)\,=\,\div \u$, that the pair of equations
\[
\bsi\,=\,\Xi[\u,p]\,:=\,\e(\u)\,-\,p\,\I \qan
\div\u\,=\,0 \qin \Omega^-\,\cup\,\Omega^+\,,
\]
is equivalent to
\begin{equation}\label{equivalent}
\bsi\,=\,\Xi[\u,p]\,:=\,\e(\u)\,-\,p\,\I \qan
p\,+\,\frac{1}{3}\,\tr \bsi\,=\,0 \qin \Omega^-\,\cup\,\Omega^+\,,
\end{equation}
which can be rewritten as
\begin{equation}\label{equivalent-otra}
\bsi^{\tt d}\,=\,\e(\u) \qan
p\,+\,\frac{1}{3}\,\tr \bsi\,=\,0 \qin \Omega^-\,\cup\,\Omega^+\,.
\end{equation}

\medskip
Consequently, from now on we replace our transmission problem \eqref{transmission-problem} by the following:
\begin{equation}\label{transmission-problem-con-presion}
\begin{array}{c}
\disp
\bsi^{\tt d}\,=\,\e(\u) \qan p\,+\,\frac{1}{3}\,\tr \bsi\,=\,0 \qin \Omega^-\,\cup\,\Omega^+\,,\\[2ex]
\disp
\bdiv\bsi \,=\,-\,\f\qin \Omega^-\,, \quad \bdiv\bsi \,=\,\0 \qin \Omega^+\,, \quad \hbox{\sc BC} \qon \Gamma_0\,, \\[2ex]
\disp
[\u]\,:=\,\u^-\,-\,\u^+\,=\,\0 \qan [\bsi \bnu ]\,:=\,(\bsi \bnu )^-\,-\,(\bsi \bnu )^+\,=\,\0 \qon\Gamma\,,\\[2ex]
\disp
\u(\x)\,=\,O(\|\x\|^{-1}) \qan p(\x)\,=\,O(\|\x\|^{-2})\quad\hbox{as}\quad \|\x\|\to +\infty\,.
\end{array}
\end{equation}

\medskip
Our aim throughout the following sections is to introduce and analyze 
several weak formulations of \eqref{transmission-problem-con-presion}, employing
either the Johnson \& N\'ed\'elec (see \cite{jn1980}) or the 
Costabel \& Han (see \cite{c1987}, \cite{h1990}) coupling procedures, and taking into account
the specific boundary condition on $\Gamma_0$. Since the pressure can be computed in terms of the stress, 
we focus mainly on the approaches that do not include $p$ as an explicit unknown but only 
as part of $\bsi$. In what follows we let $\gamma^- : \mathbf H^1(\Omega^-) \,\to\,\mathbf H^{1/2}(\partial\Omega^-)$ 
and $\gamma^-_\bnu : \mathbb H(\bdiv;\Omega^-) \,\to\,\mathbf H^{-1/2}(\partial\Omega^-)$  be 
the usual trace and normal trace operators, respectively, on $\partial\Omega^-\,:=\,\Gamma_0\,\cup\,\Gamma$.
Similarly, given a fixed Lipschitz-continuous surface $\Gamma^+$ whose interior region contains 
$\bar\Omega_0\,\cup\,\bar\Omega^-$, we let $\Omega^{++}$ be the annular domain bounded by $\Gamma$ and $\Gamma^+$,
and let $\gamma^+ : \mathbf H^1(\Omega^{++}) \,\to\,\mathbf H^{1/2}(\Gamma)$ 
and $\gamma^+_\bnu : \mathbb H(\bdiv;\Omega^{++}) \,\to\,\mathbf H^{-1/2}(\Gamma)$  be 
the usual trace and normal trace operators, respectively, on $\Gamma$. In this way, 
the transmission conditions on $\Gamma$ can be rewritten in \eqref{transmission-problem-con-presion} as:
\begin{equation}\label{transmission-conditions}
\gamma^-(\u)\,\,=\,\,\gamma^+(\u) \qan \gamma^-_\bnu(\bsi)\,\,=\,\,\gamma^+_\bnu(\bsi) \qon \Gamma\,.
\end{equation}
\section{The dual-mixed formulations in $\Omega^-$}\label{section3}
We first proceed similarly as for the linear elasticity problem (see, e.g. \cite{abd-JJAM-1984}, 
\cite{g-SPRINGER-2014}, \cite{s1988}) and introduce 
in the bounded domain $\Omega^-$ the vorticity 
\begin{equation}\label{vorticity}
\bch\,:=\,\frac12\left( \nabla \u - (\nabla \u)^{\tt t}\right) \,\in\,\mathbb L^2_{\tt skew}(\Omega^-)
\end{equation}
as an auxiliary unknown, where
\[
\mathbb L^2_{\tt skew}(\Omega^-)\,:=\,\Big\{\,\bet\,\in\,\mathbb L^2(\Omega^-):
\quad \bet^{\tt t}\,=\,-\,\bet\,\Big\}\,.
\]
In this way, the constitutive equation relating $\u$ and $\bsi$ in $\Omega^-$ becomes
\[
\bsi^{\tt d}\,\,=\,\,\nabla\u\,-\,\bch \qin \Omega^-\,,
\]
which, multiplying (tensor product $:$) by $\bta\,\in\,\mathbb H(\bdiv;\Omega^-)$ and integrating by parts, yields
\begin{equation}\label{weak-1}
\int_{\Omega^-} \bsi^{\tt d} : \bta^{\tt d}\,-\,\langle\gamma^-_\bnu(\bta),\bva\rangle_\Gamma
\,+\,\int_{\Omega^-}\u\cdot\bdiv\,\bta\,+\,\int_{\Omega^-} \bch:\bta\,+\,
\langle \gamma^-_\bnu(\bta),\gamma^-(\u)\rangle_{\Gamma_0}\,=\,0\,,
\end{equation}
where
\begin{equation}\label{trace-u}
\bva\,:=\,\gamma^-(\u)\,\,=\,\,\gamma^+(\u)\,\in\,\mathbf H^{1/2}(\Gamma)
\end{equation}
is an additional unknown, and, given $S\,\in\,\{\Gamma,\Gamma_0\}$, $\,\langle\cdot,\cdot\rangle_S$ 
denotes the duality pairing between $\mathbf H^{-1/2}(S)$ and $\mathbf H^{1/2}(S)$ with respect to the 
$\mathbf L^2(S)$-inner product. On the other hand, incorporating the equilibrium equation in $\Omega^-$ and
the symmetry of the stress tensor $\bsi$ in a weak sense, we arrive at
\begin{equation}\label{weak-2}
\int_{\Omega^-} \bv\cdot\bdiv\,\bsi\,+\,\int_{\Omega^-} \bet:\bsi\,=\,-\,\int_{\Omega^-}\f\cdot\bv
\qquad\forall\,(\bv,\bet)\,\in\,\mathbf L^2(\Omega^-) \times \mathbb L^2_{\tt skew}(\Omega^-)\,.
\end{equation}
\subsection{Dirichlet boundary condition on $\Gamma_0$}\label{section31}
We assume here that {\sc BC} on $\Gamma_0$ is given by the natural boundary condition
\begin{equation}\label{Dirichlet}
\gamma^-(\u)\,\,=\,\,\g_{_D}\,\in\,\mathbf H^{1/2}(\Gamma_0)\,,
\end{equation}
whence \eqref{weak-1} becomes
\begin{equation}\label{weak-1-Dirichlet}
\int_{\Omega^-} \bsi^{\tt d} : \bta^{\tt d}\,-\,\langle\gamma^-_\bnu(\bta),\bva\rangle_\Gamma
\,+\,\int_{\Omega^-}\u\cdot\bdiv\,\bta\,+\,\int_{\Omega^-} \bch:\bta\,=\,
-\,\langle \gamma^-_\bnu(\bta),\g_{_D}\rangle_{\Gamma_0}
\end{equation}
for each $\,\bta\,\in\, \mathbb H(\bdiv;\Omega^-)$. 

\medskip
Then, we introduce the spaces
\begin{equation}\label{spaces-D}
\mathbf X_D\,:=\,\mathbb H(\bdiv;\Omega^-)\times\mathbf H^{1/2}(\Gamma)\,,
\quad \mathbf Y_D\,:=\, \mathbf L^2(\Omega^-)\times \mathbb L^2_{\tt skew}(\Omega^-) \,,
\end{equation}
endowed with the product norms, and let $\mathbf a_D\,:\, \X_D\times\X_D \to \R$ 
and $\mathbf b_D \,:\,\X_D \times \Y_D \to \R$ be the bilinear forms given by
\begin{equation}\label{a-D}
\mathbf a_D((\bsi,\bva),(\bta,\bps))\,\,:=\,\,\int_{\Omega^-} \bsi^{\tt d} : \bta^{\tt d}
\,-\,\langle\gamma^-_\bnu(\bta),\bva\rangle_\Gamma
\quad\forall\,((\bsi,\bva),(\bta,\bps))\,\in\,\X_D\times\X_D\,,
\end{equation}
and
\begin{equation}\label{b-D}
\mathbf b_D((\bta,\bps),(\bv,\bet))\,:=\,
\int_{\Omega^-} \bv \cdot \bdiv\,\bta\,+\,\int_{\Omega^-} \bet : \bta \qquad
\forall\,((\bta,\bps),(\bv,\bet))\,\in\,\X_D\times \Y_D\,.
\end{equation}
Also, we let $\F_D:\X_D\to \R$ and $\G_D:\Y'_D\to \R$ be the linear functionals given by the
right hand side of \eqref{weak-1-Dirichlet} and \eqref{weak-2}, respectively, that is
\begin{equation}\label{F-D}
\F_D(\bta,\bps)\,:=\,-\,\langle \gamma^-_\bnu(\bta),\g_{_D}\rangle_{\Gamma_0}
\qquad
\mbox{and}
\qquad
\G_D(\bv,\bet)\,:=\,-\,\int_{\Omega^-}\f\cdot\bv.
\end{equation}
Then, collecting \eqref{weak-1-Dirichlet} and \eqref{weak-2}, we find that
the dual-mixed formulation in $\Omega^-$ can be stated as: 
Find $((\bsi,\bva),(\u,\bch))\,\in\,\X_D\times\Y_D$ such that
\begin{equation}\label{dual-mixed-formulation-Dirichlet}
\begin{array}{rcl}
\mathbf a_D((\bsi,\bva),(\bta,\bps))\,+\,\mathbf b_D((\bta,\bps),(\u,\bch)) &=& \F_D(\bta,\bps)
\qquad\forall\,(\bta,\bps)\in \X_D\,,\\[2ex]
\mathbf b_D((\bsi,\bva),(\bv,\bet)) &=& \G_D(\bv,\bet)
\qquad\forall\,(\bv,\bet)\in \Y_D\,.
\end{array}
\end{equation}
Note, however, that the above is clearly an incomplete variational formulation since
it  actually concerns four unknowns satisfying only three independent equations. In
other words,  though the bilinear forms $\a_D$ and $\b_D$ are originally defined in
$\X_D \times \X_D$ (cf. \eqref{a-D}) and $\X_D\times\Y_D$ (cf. \eqref{b-D}),
respectively, the first equation in  \eqref{dual-mixed-formulation-Dirichlet} does not
really involve the test function $\bps$. In Sections \ref{section4} and
\ref{section5} below we complete this formulation through the application of  the
boundary integral equation method in the unbounded exterior domain $\Omega^+$.
\subsection{Non-homogeneous Neumann boundary condition on $\Gamma_0$}\label{section312}
We assume now that {\sc BC} on $\Gamma_0$ is given by the essential boundary condition
\begin{equation}\label{Neumann}
\gamma^-_\bnu(\bsi)\,\,=\,\,\g_{_N}\,\in\,\mathbf H^{-1/2}(\Gamma_0)\,,
\end{equation}
which is imposed weakly as
\begin{equation}\label{weak-Neumann}
\langle \gamma^-_\bnu(\bsi),\bxi\rangle_{\Gamma_0}\,\,=\,\,\langle \g_{_N},\bxi\rangle_{\Gamma_0}\qquad\forall\,
\bxi\,\in\,\mathbf H^{1/2}(\Gamma_0)\,.
\end{equation}
Then, introducing the further unknown
\begin{equation}\label{lambda}
\bla\,:=\,\gamma^-(\u)\,\in\,\mathbf H^{1/2}(\Gamma_0)\,,
\end{equation}
we find that \eqref{weak-1} becomes
\begin{equation}\label{weak-1-Neumann}
\int_{\Omega^-} \bsi^{\tt d} : \bta^{\tt d}\,-\,\langle\gamma^-_\bnu(\bta),\bva\rangle_\Gamma
\,+\,\int_{\Omega^-}\u\cdot\bdiv\,\bta\,+\,\int_{\Omega^-} \bch:\bta\,+\,
\langle \gamma^-_\bnu(\bta),\bla\rangle_{\Gamma_0}\,=\,0
\end{equation}
for each $\,\bta\,\in\, \mathbb H(\bdiv;\Omega^-)$,
whence $\bla$ constitutes the Lagrange multiplier associated with \eqref{weak-Neumann}.

\medskip
Next, we let $\X_N\,=\,\X_D$ (cf. \eqref{spaces-D}), $\a_N\,=\,\mathbf a_D$ (cf. \eqref{a-D}), 
define the space 
\begin{equation}\label{spaces-N}
\mathbf Y_N\,:=\, \mathbf L^2(\Omega^-)\times \mathbb L^2_{\tt skew}(\Omega^-) \times \mathbf H^{1/2}(\Gamma_0)\,,
\end{equation}
endowed with the product norm, and introduce the bilinear form $\mathbf b_N \,:\,\X_N \times \Y_N \to \R$ given by
\begin{equation}\label{b-N}
\mathbf b_N((\bta,\bps),(\bv,\bet,\bxi))\,:=\,
\int_{\Omega^-} \bv \cdot \bdiv\,\bta\,+\,\int_{\Omega^-} \bet : \bta \,+\, 
\langle \gamma^-_\bnu(\bta),\bxi\rangle_{\Gamma_0}
\end{equation}
for each $\,((\bta,\bps),(\bv,\bet,\bxi))\,\in\,\X_N\times \Y_N$.
Also, we let $\F_N:\X_N\to \R$ and $\G_N:\Y_N\to \R$ be the linear functionals given by
\begin{equation}\label{F-N}
\F_N(\bta,\bps)\,:=\,0 
\qquad
\mbox{and}
\qquad
\G_N(\bv,\bet,\bxi)\,:=\,-\,\int_{\Omega^-}\f\cdot\bv
\,+\, \langle \g_{_N},\bxi\rangle_{\Gamma_0}.
\end{equation}
Then, collecting \eqref{weak-1-Neumann}, \eqref{weak-Neumann}, and \eqref{weak-2},
we find that the dual-mixed formulation in $\Omega^-$ can be stated as: 
Find $((\bsi,\bva),(\u,\bch,\bla))\,\in\,\X_N\times\Y_N$ such that
\begin{equation}\label{dual-mixed-formulation-Neumann}
\begin{array}{rcl}
\mathbf a_N((\bsi,\bva),(\bta,\bps))\,+\,\mathbf b_N((\bta,\bps),(\u,\bch,\bla)) &=& \F_N(\bta,\bps)
\qquad\forall\,(\bta,\bps)\in \X_N\,,\\[2ex]
\mathbf b_N((\bsi,\bva),(\bv,\bet,\bxi)) &=& \G_N(\bv,\bet,\bxi)
\qquad\forall\,(\bv,\bet,\bxi)\in \Y_N\,.
\end{array}
\end{equation}

\subsection{Homogeneous Neumann boundary condition on $\Gamma_0$}\label{section313}
When the Neumann boundary condition \eqref{Neumann} is homogeneous, 
that is if $\g_{_N}\,=\,\0$, then there is no need
of including the additional unknown $\bla$ (cf. \eqref{lambda}). 
In fact, we just introduce the space
\begin{equation}\label{hbdiv-0}
\mathbb H_0(\bdiv;\Omega^-)\,:=\,\Big\{\,\bta\,\in\,\mathbb H(\bdiv;\Omega^-): \quad
\gamma^-_\bnu(\bta)\,=\,\0 \qon \Gamma_0 \,\Big\}\,,
\end{equation}
and redefine $\X_N$, $\Y_N$, and $\b_N$, respectively, as
\begin{equation}\label{spaces-N-0}
\mathbf X_N\,:=\,\mathbb H_0(\bdiv;\Omega^-)\times\mathbf H^{1/2}(\Gamma)\,,
\quad \mathbf Y_N\,:=\, \mathbf L^2(\Omega^-)\times \mathbb L^2_{\tt skew}(\Omega^-) \,,
\end{equation}
and
\begin{equation}\label{b-N-0}
\mathbf b_N((\bta,\bps),(\bv,\bet))\,:=\,
\int_{\Omega^-} \bv \cdot \bdiv\,\bta\,+\,\int_{\Omega^-} \bet : \bta
\qquad\forall\,((\bta,\bps),(\bv,\bet))\,\in\,\X_N\times \Y_N\,.
\end{equation}
Then, letting $\F_N:\X_N\to \R$ and $\G_N:\Y_N\to \R$ be the linear functionals given by
\begin{equation}\label{F-N-0}
\F_N(\bta,\bps)\,:=\,0 
\qquad
\mbox{and}
\qquad
\G_N(\bv,\bet)\,:=\,-\,\int_{\Omega^-}\f\cdot\bv,
\end{equation}
we find in this case that the dual-mixed formulation in $\Omega^-$ becomes: 
Find $((\bsi,\bva),(\u,\bch))\,\in\,\X_N\times\Y_N$ such that
\begin{equation}\label{dual-mixed-formulation-Neumann-0}
\begin{array}{rcl}
\mathbf a_N((\bsi,\bva),(\bta,\bps))\,+\,\mathbf b_N((\bta,\bps),(\u,\bch)) &=& \F_N(\bta,\bps)
\qquad\forall\,(\bta,\bps)\in \X_N\,,\\[2ex]
\mathbf b_N((\bsi,\bva),(\bv,\bet)) &=& \G_N(\bv,\bet)
\qquad\forall\,(\bv,\bet)\in \Y_N\,.
\end{array}
\end{equation}
An analogue remark to the one given at the end of Section \ref{section31} is
valid here. In fact, it is clear that \eqref{dual-mixed-formulation-Neumann}
and \eqref{dual-mixed-formulation-Neumann-0} constitute incomplete 
variational formulations since they concern five and four unknowns
satisfying only four and three independent equations, respectively.   Hence,
similarly as we did for \eqref{dual-mixed-formulation-Dirichlet}, we
now announce that \eqref{dual-mixed-formulation-Neumann}
and \eqref{dual-mixed-formulation-Neumann-0} will also be completed in Sections
\ref{section4} and \ref{section5} by applying the boundary integral
equation method in the unbounded exterior domain $\Omega^+$.
\section{The boundary integral equation approach in $\Omega^+$}\label{section4}
We begin by recalling from \eqref{transmission-problem-con-presion} that in $\Omega^+$
there hold the homogeneous Stokes equations with decay conditions at infinity given by
\[
\u(\x)\,=\,O(\|\x\|^{-1}) \qan p(\x)\,=\,O(\|\x\|^{-2})\quad\hbox{as}\quad \|\x\|\to +\infty\,.
\]
Hence, following \cite[Chapter 2]{hw2008}, our aim in this section is to
apply the Green's representation formulae to express
the velocity $\u$ and pressure $p$ of the fluid in $\Omega^+$ in terms of the Cauchy 
data on $\Gamma$. For this purpose, we first let $E$ and $Q$ be the fundamental velocity
tensor and its associated pressure vector, respectively, which, using
that $\mu = 1/2$, become (see \cite[eq. (2.3.10)]{hw2008}):
\begin{equation}\label{E}
E(\x,\y)\,:=\,\frac{1}{4\,\pi}\,\left\{\,
\frac{1}{\|\x-\y\|}\,\I\,+\,\frac{(\x-\y)\,(\x-\y)^{\tt t}}{\|\x-\y\|^3}\,\right\}
\qquad\forall\,\x\not=\y\,,
\end{equation}
and
\begin{equation}\label{Q}
Q(\x,\y)\,:=\,\frac{1}{4\,\pi}\,\nabla_\y\left(\frac{1}{\|\x-\y\|}\right)
\qquad\forall\,\x\not=\y\,.
\end{equation}
In addition, we let $(\S,\D)$ and $(\Phi,\Pi)$ be the pairs of simple and double layer 
hydrodynamic potentials for the velocity and the pressure, respectively, that is
\begin{equation}\label{S}
\S\,\brh(\x)\,:=\,\int_\Gamma E(\x,\y)\,\brh(\y)\,d\s_\y\qquad\forall\,\x\,\not\in\,\Gamma\,, \quad
\forall\,\brh\,\in\,\mathbf H^{-1/2}(\Gamma)\,,
\end{equation}
\begin{equation}\label{D}
\begin{array}{c}
\disp
\D\,\bps\,:=\,\left(\begin{array}{c}
\D_1\,\bps\\[1ex] \D_2\,\bps\\[1ex] \D_3\,\bps\end{array}\right)\,, \quad
\D_i\,\bps(\x)\,:=\, \int_\Gamma\Big\{\Xi[E_i(\x,\cdot),-Q_i(\x,\cdot)](\y)\,\bnu(\y)\Big\}^{\tt t}\,\bps(\y)\,d\s_\y\\[2ex]
\disp
\forall\,\x\,\not\in\,\Gamma\,, \qquad \forall\,\bps\,\in\,\mathbf H^{1/2}(\Gamma)\,,
\end{array}
\end{equation}
where $E_i(\x,\y)$ is the $i$-th column of $E(\x,\y)$ and $Q_i(\x,\y)$ is the $i$-th component of $Q(\x,\y)$,
\begin{equation}\label{Phi}
\Phi\,\brh(\x)\,:=\,\int_\Gamma Q(\x,\y)\cdot\brh(\y)\,d\s_\y\qquad\forall\,\x\,\not\in\,\Gamma\,, \quad
\forall\,\brh\,\in\,\mathbf H^{-1/2}(\Gamma)\,,
\end{equation}
and
\begin{equation}\label{Pi}
\Pi\,\bps(\x)\,:=\,\int_\Gamma \nabla_\y\,Q(\x,\y)\,\bnu(\y)\cdot\bps(\y)\,d\s_\y\qquad
\forall\,\x\,\not\in\,\Gamma\,, \qquad \forall\,\bps\,\in\,\mathbf H^{1/2}(\Gamma)\,.
\end{equation}

\medskip
It is important to recall here that, for each $\brh\,\in\,\mathbf H^{-1/2}(\Gamma)$ 
and for each $\bps\,\in\,\mathbf H^{1/2}(\Gamma)$, the velocity/pressure pairs $(\S\,\brh,\Phi\,\brh)$ 
and $(\D\,\bps,\Pi\,\bps)$ satisfy the homogeneous Stokes equations in $\R^3 \backslash \Gamma$.
In addition, the main continuity properties of $\S$, $\D$, $\Phi$, and $\Pi$  are summarized 
in the following lemma.
\begin{lemma}\label{properties-1}
The hydrodynamic potentials define the following bounded linear operators:
\[
\begin{array}{lcl}
\S \,:\, \mathbf H^{-1/2}(\Gamma) &\to& \mathbf H^1_{\mathrm{div}}(\Omega^-;\Delta)
\times \mathbf H^1_{\mathrm{div,loc}}(\Omega^+,\Delta)\,,\\[2ex]
\D \,:\, \mathbf H^{1/2}(\Gamma) &\to& \mathbf H^1_{\mathrm{div}}(\Omega^-;\Delta)
\times \mathbf H^1_{\mathrm{div,loc}}(\Omega^+,\Delta)\,,\\[2ex]
\Phi \,:\, \mathbf H^{-1/2}(\Gamma) &\to& L^2(\Omega^-) \times L^2(\Omega^+)\,,\\[2ex]
\Pi \,:\, \mathbf H^{1/2}(\Gamma) &\to& L^2(\Omega^-) \times L^2(\Omega^+)\,,
\end{array}
\]
where
\[
\mathbf H^1_{\mathrm{div}}(\Omega^-;\Delta)\,:=\,\big\{\,\bv\,\in\,\mathbf H^1(\Omega^-):
\quad \div\,\bv\,=\,0 \qin \Omega^- \qan \Delta\,\bv\,\in\,\tilde{\mathbf H}^{-1}_0(\Omega^-)\,\big\}\,,
\]
\[
\mathbf H^1_{\mathrm{div,loc}}(\Omega^+,\Delta)\,:=\,\big\{\,\bv\,\in\,
\mathbf H^1_{\mathrm{loc}}(\Omega^+): \quad \div\,\bv\,=\,0 \qin \Omega^+ \qan 
\Delta\,\bv\,\in\,\big(\mathbf H^1_{\mathrm{loc}}(\Omega^+)\big)'\,\big\}\,,
\]
and $\tilde{\mathbf H}^{-1}_0(\Omega^-)$ is the orthogonal complement in $\big(\mathbf H^1(\Omega^-)\big)'$
of $\big\{\,\bv\,\in\,\big(\mathbf H^1(\Omega^-)\big)': \quad \mathrm{supp}\,\bv\,\subseteq\,\Gamma
\Big\}$. 
\end{lemma}
\begin{proof}
It follows by combining analogue continuity properties for the Lam\'e system and the Laplacian. 
We refer to \cite[Lemmas 5.6.4 and 5.6.6]{hw2008} and \cite[Theorem 3.3]{kw2006} for details. Alternatively, \cite{SelgasFJS} contains similar boundedness statements using weighted Sobolev spaces.
\end{proof}

\bigskip
We now let $\V$, $\K$, $\K^{\tt t}$, and $\W$ be the boundary integral operators of
the simple, double, adjoint of the double, and hypersingular layer hydrodynamic potentials,
respectively.  These operators can be defined using lateral traces of the single and double layer hydrodynamic potentials (see  \cite[Lemma 5.6.5]{hw2008} and \cite[Sections 5 and 6]{SelgasFJS}:
\begin{equation}\label{trace-1}
\gamma^+(\S\,\brh)\,=\,\gamma^-(\S\,\brh)\,=\,\V\,\brh\qquad\forall\,\brh\,\in\,\mathbf H^{-1/2}(\Gamma)\,,
\end{equation}
\begin{equation}\label{trace-2}
\gamma^\pm(\D\,\bps)\,=\,\left(\pm\,\frac12\,\I\,+\,\K\right)\,\bps \qquad\forall\,\bps\,\in\,\mathbf H^{1/2}(\Gamma)\,,
\end{equation}
\begin{equation}\label{trace-3}
\gamma^\pm_\bnu\big(\Xi[\S\,\brh,\Phi\,\brh]\big)\,=\,\left(\mp\,\frac12\,\I\,+\,\K^{\tt t}\right)\,\brh
\qquad\forall\,\brh\,\in\,\mathbf H^{-1/2}(\Gamma)\,,
\end{equation}
\begin{equation}\label{trace-4}
\gamma^+_\bnu\big(\Xi[\D\,\bps,\Pi\,\bps]\big)\,=\,\gamma^-_\bnu\big(\Xi[\D\,\bps,\Pi\,\bps]\big)
\,=\,-\,\W\,\bps \qquad\forall\,\bps\,\in\,\mathbf H^{1/2}(\Gamma)\,.
\end{equation}
As a consequence of \eqref{trace-1} - \eqref{trace-4} the following jump conditions hold:
\begin{equation}\label{jump-1}
[\gamma (\S\,\brh)]\,:=\,\gamma^-(\S\,\brh) \,-\, \gamma^+(\S\,\brh)\,=\,\0 
\qquad\forall\,\brh\,\in\,\mathbf H^{-1/2}(\Gamma)\,,\\[1.5ex]
\end{equation}
\begin{equation}\label{jump-2}
[\gamma (\D\,\bps)]\,:=\,\gamma^-(\D\,\bps) \,-\, \gamma^+(\D\,\bps)\,=\,-\,\bps
\qquad\forall\,\bps\,\in\,\mathbf H^{1/2}(\Gamma)\,,\\[1.5ex]
\end{equation}
\begin{equation}\label{jump-3}
[\gamma_\bnu\big(\Xi[\S\,\brh,\Phi\,\brh]\big)\,:=\,\gamma^-_\bnu\big(\Xi[\S\,\brh,\Phi\,\brh]\big)
\,-\,\gamma^+_\bnu\big(\Xi[\S\,\brh,\Phi\,\brh]\big)\,=\,\brh
\qquad\forall\,\brh\,\in\,\mathbf H^{-1/2}(\Gamma)\,,\\[1.5ex]
\end{equation}
\begin{equation}\label{jump-4}
[\gamma_\bnu\big(\Xi[\D\,\bps,\Pi\,\bps]\big)\,:=\,\gamma^-_\bnu\big(\Xi[\D\,\bps,\Pi\,\bps]\big)
\,-\,\gamma^+_\bnu\big(\Xi[\D\,\bps,\Pi\,\bps]\big)\,=\,\0
\qquad\forall\,\bps\,\in\,\mathbf H^{1/2}(\Gamma)\,.
\end{equation}
Integral expressions for the above boundary integral operators can be found in the literature. For smooth enough densities $\brh$ and almost everywhere on $\Gamma$ we can write \cite[(2.3.15)]{hw2008}
\begin{equation}\label{V}
\V\,\brh(\x)\,=\,\S\,\brh(\x)\,:=\,\int_\Gamma E(\x,\y)\,\brh(\y)\,d\s_\y.
\end{equation}
An integral expression for $\K$ can be found in \cite[(2.3.30)]{hw2008}
\begin{equation}\label{K}
\K\,\bps(\x)\,:=\,\frac{3}{4\,\pi}\,\mathrm{p.v.}\int_{\Gamma}
\frac{\big((\x-\y)\cdot\bnu(\y)\big)\,\big((\x-\y)\cdot\bps(\y)\big)}{\|\x-\y\|^5}\,(\x-\y)\,d\s_\y.
\end{equation}
This integral formula involves a Cauchy principal value.
For the explicit integral form of $\W$ we refer to \cite[eqs. (2.3.30) and (2.3.31)]{hw2008}.

\medskip
Furthermore, some of the main properties of $\V$, $\K$, $\K^{\tt t}$, and $\W$ are collected next. 
To this end, given $\mathcal O \,\subseteq\,\R^3$ and $\ell\in\mathbb N\cup\{0\}$, we now let 
$P_\ell(\mathcal O)$ be the space of polynomials of degree $\le \ell$ on $\mathcal O$, and let
$\mathbf{RM}(\mathcal O)$ be the space of 
rigid motions in $\mathcal O$, that is
\[
\mathbf{RM}(\mathcal O)\,:=\,\big\{\,\mathbf z: \quad \mathbf z(\x)\,=\,
\mathbf c\,+\,\mathbf d\times \x\quad\forall\,\x\in\mathcal O \,; \quad 
\mathbf c,\,\mathbf d\in\R^3\big\}\,.
\]
We also introduce the vector field $\mathbf m:\Gamma \to \R^3$ given by 
\[
\mathbf m(\x):=\x - \frac1{|\Gamma|}\int_\Gamma \x \,d\s_\x,
\]
and the spaces
\[
\mathbf H^{-1/2}_0(\Gamma)\,\,:=\Big\{\,\brh\,\in\,\mathbf H^{-1/2}(\Gamma): \quad
\langle \mathbf m,\brh\rangle_\Gamma\,=\,0\Big\}
\]
and
\[
\mathbf H^{1/2}_0(\Gamma)\,\,:=\Big\{\,\bps\,\in\,\mathbf H^{1/2}(\Gamma): \quad
\langle \mathbf r,\bps\rangle_\Gamma\,=\,0\quad\forall\,\mathbf r\,\in\,\mathbf{RM}(\Gamma)\Big\}.
\]

\begin{lemma}\label{properties-2}
The following boundary integral operators are linear and bounded:
\begin{equation}\label{properties-V-K-W}
\begin{array}{lcl}
\V\,:\, \mathbf H^{-1/2}(\Gamma) &\to& \mathbf H^{1/2}(\Gamma)\,,\\[2ex]
\K\,:\, \mathbf H^{1/2}(\Gamma) &\to& \mathbf H^{1/2}(\Gamma)\,,\\[2ex]
\K^{\tt t}\,:\, \mathbf H^{-1/2}(\Gamma) &\to& \mathbf H^{-1/2}(\Gamma)\,,\\[2ex]
\W\,:\, \mathbf H^{1/2}(\Gamma) &\to& \mathbf H^{-1/2}(\Gamma)\,.
\end{array}
\end{equation}
In addition, $\V$ and $\W$ are selfadjoint, 
\begin{equation}\label{kernels-V-and-W}
\ker\left(\V\right)
\,=\,\ker\left(\frac12\,\I\,-\,\K^{\tt t}\right)
\,=\,P_0(\Gamma)\,\bnu\,, \qquad
\ker\left(\W\right)\,=\,\mathrm{ker}\left(\frac12\,\I +\,\K\right)
\,=\,\mathbf{RM}(\Gamma)\,,
\end{equation}
and there exist $\alpha_1\,, \alpha_2\,>\,0$ such that
\begin{equation}\label{V-eliptico}
\langle \brh,\V\,\brh\rangle_\Gamma 
\,\ge\,\alpha_1\,\|\brh\|^2_{-1/2,\Gamma}\qquad\forall\,\brh\,\in\,\mathbf H^{-1/2}_0(\Gamma)
\end{equation}
and
\begin{equation}\label{W-eliptico}
\langle \W\,\bps,\bps\rangle_\Gamma 
\,\ge\,\alpha_2\,\|\bps\|^2_{1/2,\Gamma}\qquad\forall\,\bps\,\in\,\mathbf H^{1/2}_0(\Gamma)\,.
\end{equation}
\end{lemma}
\begin{proof}
The proofs of these results appear in \cite[Sections 5 and 6]{SelgasFJS}.
\end{proof}

\medskip
Note that the decompositions 
\begin{equation}\label{decompositions}
\mathbf H^{-1/2}(\Gamma)\,=\,\mathbf H^{-1/2}_0(\Gamma)\,\oplus\,{P}_0(\Gamma)\,\bnu \qan
\mathbf H^{1/2}(\Gamma)\,=\,\mathbf H^{1/2}_0(\Gamma)\,\oplus\,\mathbf{RM}(\Gamma)\,,
\end{equation}
are stable and have associated oblique projectors $\boldsymbol\pi_{_\nu} : \mathbf H^{-1/2}_0(\Gamma) \to P_0(\Gamma)\bnu$ and $\boldsymbol\pi_{_{RM}} :\mathbf H^{1/2}(\Gamma) \to \mathbf{RM}(\Gamma)$. Therefore, 
the inequalities \eqref{V-eliptico} and \eqref{W-eliptico} are equivalent to
\begin{equation}\label{V-eliptico-equiv}
\langle \brh,\V\,\brh\rangle_\Gamma 
\,\ge\,\tilde\alpha_1\,\|\brh-\boldsymbol\pi_{_\nu} \brh\|^2_{-1/2,\Gamma}
\qquad\forall\,\brh\,\in\,\mathbf H^{-1/2}(\Gamma)
\end{equation}
and
\begin{equation}\label{W-eliptico-equiv}
\langle \W\,\bps,\bps\rangle_\Gamma 
\,\ge\,\tilde\alpha_2\,\|\bps-\boldsymbol\pi_{_{RM}}\bps\|^2_{1/2,\Gamma}\qquad\forall\,\bps\,\in\,\mathbf H^{1/2}(\Gamma).
\end{equation}
As a simple consequence of \eqref{kernels-V-and-W} we can prove that
\begin{equation}\label{new4.24}
\langle\bnu,\mathbf r\rangle_\Gamma
	=\left\langle \left( \frac12\I+\K^{\tt t}\right)\bnu,\mathbf r\right\rangle
	=\left\langle \bnu,\left( \frac12\I+\K\right)\mathbf r\right\rangle=0\qquad\forall\mathbf r\in \mathbf{RM}(\Gamma).
\end{equation}

\bigskip
On the other hand, we have the following technical result.
\begin{lemma}\label{W-psi-psi}
There holds
\begin{equation}\label{eq-W-psi-psi}
\langle \W\,\bps,\bps\rangle_\Gamma\,=\,\|\e(\D\,\bps)\|^2_{0,\R^3\backslash \Gamma}
\qquad\forall\,\bps\,\in\,\mathbf H^{1/2}(\Gamma)\,.
\end{equation}
\end{lemma}
\begin{proof}
Given $\bps\,\in\,\mathbf H^{1/2}(\Gamma)$, it follows from \eqref{trace-4} and
\eqref{jump-2} that
\[
\begin{array}{l}
\disp
\langle \W\,\bps,\bps\rangle_\Gamma\,=\,\langle -\, \gamma^\pm_\bnu\big(\Xi[\D\,\bps,\Pi\,\bps]\big),
\gamma^+(\D\,\bps) \,-\, \gamma^-(\D\,\bps) \rangle_\Gamma\\[2ex]
\disp\quad
= \quad \langle \gamma^-_\bnu\big(\Xi[\D\,\bps,\Pi\,\bps]\big),\gamma^-(\D\,\bps) \rangle_\Gamma
\,-\,\langle \gamma^+_\bnu\big(\Xi[\D\,\bps,\Pi\,\bps]\big),\gamma^+(\D\,\bps) \rangle_\Gamma\,.
\end{array}
\]
Next, integrating by parts in $\Omega\,:=\,\Omega_0 \cup \Omega^-$ and recalling that
the velocity/pressure pair $(\D\,\bps,\Pi\,\bps)$ satisfies the homogeneous Stokes equations,
we find that
\[
\begin{array}{l}
\disp
\langle \gamma^-_\bnu\big(\Xi[\D\,\bps,\Pi\,\bps]\big),\gamma^-(\D\,\bps) \rangle_\Gamma
\,=\,\int_\Omega \nabla \D\,\bps : \Xi[\D\,\bps,\Pi\,\bps]\,+\,
\int_\Omega \D\,\bps \cdot \bdiv\,\Xi[\D\,\bps,\Pi\,\bps]\\[2ex]
\disp\quad
= \quad \int_\Omega \nabla \D\,\bps : \Big\{\,\e(\D\,\bps)\,-\,\Pi\,\bps\,\I\,\Big\}
\,=\,\int_\Omega \nabla \D\,\bps : \e(\D\,\bps)\,=\,\|\e(\D\,\bps)\|^2_{0,\Omega}\,.
\end{array}
\]
Similarly, integrating by parts in $\Omega^+$, noting that $\bnu$ points inward $\Omega^+$,
and using additionally the conditions at infinity, we deduce that
\[
\,-\,\langle \gamma^+_\bnu\big(\Xi[\D\,\bps,\Pi\,\bps]\big),\gamma^+(\D\,\bps) \rangle_\Gamma
\,=\,\|\e(\D\,\bps)\|^2_{0,\Omega^+}\,,
\]
which, together with the previous identity, yields \eqref{eq-W-psi-psi} and ends the proof.
Alternatively, this proof can also be found in \cite[Proposition 6.4]{SelgasFJS}.

\end{proof}

\medskip
We now go back to the homogeneous Stokes equations in $\Omega^+$. In fact,
according to the Green's formulae provided in \cite[Section 2.3.1]{hw2008}, we have
the representations
\begin{equation}\label{Green-u}
\u\,=\,-\,\S\,\gamma^+_\bnu(\bsi)\,+\,\D\,\gamma^+(\u) \qin \Omega^+\,,
\end{equation}
and
\begin{equation}\label{Green-p}
p\,=\,-\,\Phi\,\gamma^+_\bnu(\bsi)\,+\,\Pi\,\gamma^+(\u)\qin \Omega^+\,.
\end{equation}
Therefore, evaluating the operators $\gamma^+$ and $\gamma^+_\bnu$ in $\u$ and $\bsi\,=\,\Xi[\u,p]$, 
respectively, with $\u$ and $p$ given by \eqref{Green-u} and \eqref{Green-p}, 
and applying the trace properties,
we arrive at the following boundary integral equations:
\begin{equation}\label{bie-1}
\gamma^+(\u)\,=\,-\,\V\,\gamma^+_\bnu(\bsi)\,+\,\left(\frac12\,\I\,+\,\K \right)\,\gamma^+(\u) \qon\Gamma\,,
\end{equation}
and
\begin{equation}\label{bie-2}
\gamma^+_\bnu(\bsi)\,=\,\left(\frac12\,\I\,-\,\K^{\tt t} \right)\,\gamma^+_\bnu(\bsi)
\,-\,\W\,\gamma^+(\u) \qon\Gamma\,.
\end{equation}
Moreover, thanks to the transmission conditions \eqref{transmission-conditions}
and the introduction of the additional unknown $\bva$ (cf. \eqref{trace-u}), the 
above equations become:
\begin{equation}\label{bie-1-with-bva}
\bva\,=\,-\,\V\,\gamma^-_\bnu(\bsi)\,+\,\left(\frac12\,\I\,+\,\K \right)\,\bva \qon\Gamma\,,
\end{equation}
and
\begin{equation}\label{bie-2-with-bva}
\W\,\bva\,+\,\left(\frac12\,\I\,+\,\K^{\tt t} \right)\,\gamma^-_\bnu(\bsi)\,=\,\0 \qon\Gamma\,.
\end{equation}
\section{The coupled variational formulations}\label{section5}
In this section we combine the dual-mixed approach in $\Omega^-$ (cf. Section \ref{section3})
with the boundary integral equation method in $\Omega^+$ (cf. Section \ref{section4})
to derive and analyze coupled variational formulations
for the transmission problem \eqref{transmission-problem-con-presion}.
\subsection{J \& N coupling with homogeneous Neumann boundary conditions on $\Gamma_0$}\label{section331}
Here we follow the Johnson-N\'ed\'elec coupling method (see \cite{bj1979}, \cite{jn1980}) and incorporate the
single boundary integral equation \eqref{bie-2-with-bva} into the dual-mixed variational formulation
in $\Omega^-$ given by \eqref{dual-mixed-formulation-Neumann-0}, which considers the homogeneous Neumann boundary condition $\,\gamma_\bnu^-(\bsi)\,=\,\0$ \,on\, $\Gamma_0$. More precisely, we test \eqref{bie-2-with-bva}
against $\bps\,\in\,\mathbf H^{1/2}(\Gamma)$ and add the resulting equation to
the first equation of \eqref{dual-mixed-formulation-Neumann-0} thus yielding a redefinition
of the bilinear form $\a_N = \a_D$ (cf. \eqref{a-D}). In this way, our coupled variational 
formulation reads as follows: Find $((\bsi,\bva),(\u,\bch))\,\in\,{\X_N}\times{\Y_N}$ such that
\begin{equation}\label{dual-mixed-formulation-complete}
\begin{array}{rcl}
\mathbf a_N((\bsi,\bva),(\bta,\bps))\,+\,\mathbf b_N((\bta,\bps),(\u,\bch)) &=& \F_N(\bta,\bps)
\qquad\forall\,(\bta,\bps)\,\in\, {\X_N}\,,\\[2ex]
\mathbf b_N((\bsi,\bva),(\bv,\bet)) &=& \G_N(\bv,\bet)
\qquad\forall\,(\bv,\bet)\,\in\, {\Y_N}\,,
\end{array}
\end{equation}
where 
\begin{equation}\label{spaces-new}
\X_N\,:=\,\mathbb H_0(\bdiv;\Omega^-)\times\mathbf H^{1/2}(\Gamma)\,,
\quad \Y_N\,:=\, \mathbf L^2(\Omega^-)\times \mathbb L^2_{\tt skew}(\Omega^-) \,,
\end{equation}
$\a_N : \X_N \times \X_N \to \R$ \,and\, $\b_N : \X_N \times \Y_N \to \R$ are the bounded bilinear forms
defined by
\begin{equation}\label{a-N-complete}
\mathbf a_N((\bsi,\bva),(\bta,\bps)) := \int_{\Omega^-} \bsi^{\tt d} : \bta^{\tt d}
\,+\,\langle \W\,\bva,\bps\rangle_\Gamma\,+\,\left\langle \left(\frac12\,\I\,+\,\K^{\tt t} \right)\,
\gamma^-_\bnu(\bsi),\bps \right\rangle_\Gamma - \langle \gamma^-_\bnu(\bta),\bva\rangle_\Gamma
\end{equation}
and
\begin{equation}\label{b-N-complete}
\mathbf b_N((\bta,\bps),(\bv,\bet))\,:=\,
\int_{\Omega^-} \bv \cdot \bdiv\,\bta\,+\,\int_{\Omega^-} \bet : \bta \,,
\end{equation}
and $\F_N : \X_N \to \R$ and $\G_N : \Y_N \to \R$ are the bounded linear functionals given by
\begin{equation}\label{F-complete}
\F_N(\bta,\bps)\,:=\,0 \,
\qquad\mbox{and}\qquad
\G_N(\bv,\bet)\,:=\,-\,\int_{\Omega^-}\f\cdot\bv.
\end{equation}

\medskip
We now observe from \eqref{b-N-complete} that the bounded linear operator 
induced by $\mathbf b_N$, say $\mathbf B_N : {\X_N} \to {\Y_N}$,
is given by $\,\mathbf B_N((\bta,\bps))\,:=\,\big(\bdiv\,\bta, \frac12(\bta - \bta^{\tt t})\big)$ for any 
$(\bta,\bps)\,\in\,{\X_N}$. If follows easily that
$\mathbf V_N$, the kernel of $\mathbf B_N$, reduces to
\[
\mathbf V_N\,:=\,\Big\{\,(\bta,\bps)\,\in\,{\X_N}: \quad \bta\,=\,\bta^{\tt t} \qan \bdiv\,\bta \,=\,\0 
\qin \Omega^-\,\Big\}\,.
\]

\medskip
The following lemmas, which establish a positiveness property of $\a_N$ on $\mathbf V_N$
and an inf-sup condition for $\b_N$, are crucial for the forthcoming analysis.
\begin{lemma}\label{quasi-coerciveness-a}
There holds
\begin{equation}\label{eq-quasi-coerciveness-a}
\a_N((\bta,\bps),(\bta,\bps))\,\ge\,\frac12\,\Big\{\,\|\bta^{\tt d}\|^2_{0,\Omega^-}
\,+\,\langle \W\,\bps,\bps\rangle_\Gamma\,\Big\} \qquad\forall\,(\bta,\bps)\,\in\,\mathbf V_N\,.
\end{equation}
\end{lemma}
\begin{proof}
Given $(\bta,\bps)\,\in\,\mathbf V_N$ we have from \eqref{a-N-complete} and \eqref{trace-2}
\begin{equation}\label{previous}
\begin{array}{c}
\disp
\mathbf a_N((\bta,\bps),(\bta,\bps)) \,=\,\|\bta^{\tt d}\|^2_{0,\Omega^-}\,+\, \langle \W\,\bps,\bps\rangle_\Gamma
\,+\,\left\langle \gamma^-_\bnu(\bta),\left(\,-\,\frac12\,\I\,+\,\K \right)\,\bps \right\rangle_\Gamma\,,
\\
\disp
= \|\bta^{\tt d}\|^2_{0,\Omega^-}\,+\, \langle \W\,\bps,\bps\rangle_\Gamma
\,+\,\left\langle \gamma^-_\bnu(\bta),\gamma^-(\D\,\bps)\right\rangle_\Gamma\,.
\end{array}
\end{equation}
Hence, integrating by parts in $\Omega^-$ and using that $\gamma^-_\bnu(\bta)\,=\,\0$ on $\Gamma_0$, we find that
\begin{equation}\label{above}
\begin{array}{l}
\disp
\langle \gamma^-_\bnu(\bta),\gamma^-(\D\,\bps)\rangle_\Gamma\,=\,
\int_{\Omega^-}\Big\{\, \nabla \D\,\bps : \bta \,+\, \D\,\bps \cdot\bdiv\,\bta\,\Big\}\\[2ex]
\disp\quad
= \quad \int_{\Omega^-} \e\big(\D\,\bps\big) : \bta\,=\,\int_{\Omega^-} \e\big(\D\,\bps\big) : \bta^{\tt d}\,,
\end{array}
\end{equation}
where the free-divergence and symmetry properties of $\bta$, together
with the incompressibility  condition satisfied by $\D\,\bps$, have been
utilized in the last two equalities. In this way, replacing
\eqref{above} into \eqref{previous}, and then applying  Cauchy-Schwarz's
inequality and the identity \eqref{eq-W-psi-psi}, we deduce that
\[
\begin{array}{l}
\disp
a_N((\bta,\bps),(\bta,\bps)) \,=\,
\|\bta^{\tt d}\|^2_{0,\Omega^-}\,+\, \langle \W\,\bps,\bps\rangle_\Gamma
\,+\,\int_{\Omega^-} \e\big(\D\,\bps\big) : \bta^{\tt d}\\[2ex]
\disp\quad
\ge \quad 
\frac12\,\|\bta^{\tt d}\|^2_{0,\Omega^-}\,+\, \langle \W\,\bps,\bps\rangle_\Gamma\,-\,\frac12\,\|\e(\D\,\bps)\|^2_{0,\Omega^-}\\[2ex]
\disp\quad
\ge \quad 
\frac12\,\|\bta^{\tt d}\|^2_{0,\Omega^-}\,+\, \langle \W\,\bps,\bps\rangle_\Gamma\,-\,
\frac12\,\|\e(\D\,\bps)\|^2_{0,\R^3\backslash \Gamma}\\[2ex]
\disp\quad
= \quad 
\frac12\,\|\bta^{\tt d}\|^2_{0,\Omega^-}\,+\,\frac12\, \langle \W\,\bps,\bps\rangle_\Gamma\,,
\end{array}
\]
which finishes the proof.

\end{proof}

\begin{lemma}\label{inf-sup-b}
There exists $\beta > 0$ such that for any $(\bv,\bet)\,\in\,{\Y_N}$ there holds
\begin{equation}\label{eq-inf-sup-b}
\sup_{(\bta,\bps)\in {\X_N}\backslash\{\0\}}\,
\frac{\mathbf b_N((\bta,\bps),(\bv,\bet))}{\|(\bta,\bps)\|_{\X_N}}\,\ge\,\beta\,\|(\bv,\bet)\|_{\Y_N}\,.
\end{equation}
\end{lemma}
\begin{proof}
It reduces to show that the operator $\mathbf B_N$ is surjective. In fact, given
$(\bv,\bet)\,\in\,{\Y_N}$, we let $\z$ be the unique element in $\,\mathbf H^1_\Gamma(\Omega^-)
\,:=\,\Big\{\,\mathbf w\in \mathbf H^1(\Omega^-): \quad \mathbf w\,=\,\0 \qon \Gamma\,\Big\}\,$,
whose existence is guaranteed by the second Korn inequality, such that
\[
\int_{\Omega^-} \e(\z) : \e(\w)\,=\,-\int_{\Omega^-} \bv \cdot \w
\,- \int_{\Omega^-} \bet : \nabla \w \qquad\forall\,\w\,\in\,\mathbf H^1_\Gamma(\Omega^-)\,.
\]
Hence, defining $\,\,\widehat\bta\,:=\,\e(\z) + \bet\,\in\,\mathbb L^2(\Omega^-)$,
we deduce from the above formulation that $\,\bdiv\,\widehat\bta\,=\,\bv$ \,in\, $\Omega^-$, which shows that
$\,\widehat\bta\,\in\,\mathbb H(\bdiv;\Omega^-)$, and then that $\gamma^-_\bnu(\widehat\bta)
\,=\,\0$ \,on\, $\Gamma_0$. In this way, $\widehat \bta\,\in\,\mathbb H_0(\bdiv;\Omega^-)$
and it is easy to see that $\,\mathbf B_N((\widehat\bta,\0))\,=\,(\bv,\bet)$, which ends the proof.

\end{proof}

\medskip
Note that the fact that $\mathbf b_N((\bta,\bps),(\bv,\bet))$ does not depend on $\bps$ guarantees
that the inf-sup condition \eqref{eq-inf-sup-b} can also be rewritten as
\begin{equation}\label{eq-inf-sup-b-simple}
\sup_{\bta\in \mathbb H_0(\bdiv;\Omega^-)\backslash\{\0\}}\,
\frac{\mathbf b_N((\bta,\0),(\bv,\bet))}{\|(\bta,\0)\|_{\X_N}}\,\ge\,\beta\,\|(\bv,\bet)\|_{\Y_N}
\qquad\forall\,(\bv,\bet)\in{\Y_N}\,.
\end{equation}

\medskip
We now begin the solvability analysis of \eqref{dual-mixed-formulation-complete} by identifying
previously the solutions of the associated homogeneous problem.
\begin{lemma}\label{kernel}
The set of solutions of the homogeneous version of \eqref{dual-mixed-formulation-complete}
is given by
\[
\Big\{\,\big((\bsi,\bva),(\u,\bch)\big)\,:=\,
\big((\0,\z|_\Gamma),(\z,\nabla \z)\big): \qquad \z\,\in\,\mathbf{RM}(\Omega^-)\,\Big\}\,.
\]
\end{lemma}
\begin{proof}
Let $\big((\bsi,\bva),(\u,\bch)\big)\,\in\,{\X_N} \times {\Y_N}$ be a solution of 
\eqref{dual-mixed-formulation-complete} with $\f\,=\,\0$. It is clear from the 
second equation that $(\bsi,\bva)\,\in\,\mathbf V_N$,
\,that is\, $\bsi\,=\,\bsi^{\tt t}$ \,and\, $\bdiv \bsi\,=\,\0$ \,in\, $\Omega^-$.
Then, taking in particular $(\bta,\bps)\,=\,(\bsi,\bva)$ in the first equation,
and then applying the inequalities \eqref{eq-quasi-coerciveness-a} (cf. Lemma 
\ref{quasi-coerciveness-a}) and \eqref{W-eliptico-equiv}, we find that
\[
0\,=\,\mathbf a_N\big((\bsi,\bva),(\bsi,\bva)\big)\,\ge\,\frac12\,
\Big\{\,\|\bsi^{\tt d}\|^2_{0,\Omega^-}\,+\,\langle \W\,\bva,\bva\rangle_\Gamma\,\Big\}
\,\ge\,\frac12\,\|\bsi^{\tt d}\|^2_{0,\Omega^-}\,+\,\frac{\tilde\alpha_2}{2}\,
\|\bva-\boldsymbol\pi_{_{RM}}\bva\|^2_{1/2,\Gamma}\,,
\]
which gives $\bsi^{\tt d}\,=\,\0$ \,in\, $\Omega^-$ \,and $\bva\,=\,\mathbf z|_{\Gamma}$, \,with $\,\mathbf z\,\in\,\mathbf{RM}(\Omega^-)$.
In turn, the conditions satisfied by $\bsi$, namely $\bdiv \bsi = \0$
\,and\, $\bsi^{\tt d} = \0$ in $\Omega^-$, together with the fact that 
$\gamma^-_\bnu(\bsi) = \0$ on $\Gamma_0$ imply that $\bsi \,=\,\0$. Next,
taking $\bps = \0$ in the first equation of our homogeneous problem, and then
integrating by parts in $\Omega^-$, we obtain that for any $\bta\in \mathbb H_0(\bdiv;\Omega^-)$ 
there holds
\[
\begin{array}{l}
\disp
0\,=\,\mathbf b_N((\bta,\0),(\u,\bch))\,-\,\langle \gamma^-_\bnu(\bta),\bva\rangle_\Gamma
\,=\,\mathbf b_N((\bta,\0),(\u,\bch))\,-\,\langle \gamma^-_\bnu(\bta),\z\rangle_\Gamma\\[2ex]
\disp\,\,\,\,\,
=\,\mathbf b_N((\bta,\0),(\u,\bch))\,-\,\int_{\Omega^-} \z\cdot\bdiv \bta
\,-\,\int_{\Omega^-} \nabla\z : \bta\,=\,
\mathbf b_N((\bta,\0),(\u-\z,\bch-\nabla \z))\,,
\end{array}
\] 
which, thanks to the inf-sup condition \eqref{eq-inf-sup-b-simple}, gives
$(\u,\bch)\,=\,(\z,\nabla\z)$. Conversely, it is easy to see, in particular
using that $\mathrm{ker}\,(\W)\,=\,\mathbf{RM}(\Gamma)$ (cf. \eqref{kernels-V-and-W}), 
that for any $\z\,\in\,\mathbf{RM}(\Omega^-)$ the element $\big((\bsi,\bva),(\u,\bch)\big)\,:=\,
\big((\0,\z|_\Gamma),(\z,\nabla \z)\big)$ solves the homogeneous version 
of \eqref{dual-mixed-formulation-complete}.

\end{proof}

\medskip
According to the above lemma and the decomposition
$\mathbf H^{1/2}(\Gamma)\,=\,\mathbf H^{1/2}_0(\Gamma)\,\oplus\,\mathbf{RM}(\Gamma)$ (cf. \eqref{decompositions}), 
and in order to guarantee the unique solvability of the coupled problem \eqref{dual-mixed-formulation-complete}, 
we now look for the solution $((\bsi,\bva),(\u,\bch))$ in the space ${\widetilde{\X}_N}\times{\Y_N}$, where
\begin{equation}\label{space-X-0}
\widetilde{\X}_N\,:=\,\mathbb H_0(\bdiv;\Omega^-)\times\mathbf H^{1/2}_0(\Gamma)\,.
\end{equation}
In turn, it is easy to see, using that $\langle \W\,\bva,\bps\rangle_\Gamma\,=\, 
\langle \W\,\bps,\bva\rangle_\Gamma$ and that $\disp \mathrm{ker}(\W)\,=\,
\mathrm{ker}\left(\frac12\,\I +\,\K\right)\,=\,
\mathbf{RM}(\Gamma)$ (cf. \eqref{kernels-V-and-W}),  that the ocurrence of the first equation of 
\eqref{dual-mixed-formulation-complete} can be equivalently established for any 
$(\bta,\bps) \in{\widetilde{\X}_N}$. As a consequence, and instead of \eqref{dual-mixed-formulation-complete},
we now seek $((\bsi,\bva),(\u,\bch))\,\in\,{\widetilde{\X}_N}\times{\Y_N}$ such that
\begin{equation}\label{dual-mixed-formulation-complete-new}
\begin{array}{rcl}
\mathbf a_N((\bsi,\bva),(\bta,\bps))\,+\,\mathbf b_N((\bta,\bps),(\u,\bch)) &=& \F_N(\bta,\bps)
\qquad\forall\,(\bta,\bps)\,\in\, {\widetilde{\X}_N}\,,\\[2ex]
\mathbf b_N((\bsi,\bva),(\bv,\bet)) &=& \G_N(\bv,\bet)
\qquad\forall\,(\bv,\bet)\,\in\, {\Y_N}\,,
\end{array}
\end{equation}

\medskip
The following two lemmas are needed to show the well-posedness of \eqref{dual-mixed-formulation-complete-new}. 
They make use of the decomposition defined by \eqref{tilde-hbdiv} and \eqref{decomp-hdiv-1},
which says in this case that each $\bta\in \mathbb H(\bdiv;\Omega^-)$ can be written 
in a unique way as $\bta = \bta_0 + d\,\I$,
with $\bta_0\in \widetilde{\mathbb H}(\bdiv;\Omega^-)$ and $d\in \R$. The associated projector will be denoted $\boldsymbol\pi_{_I}:\mathbb H(\bdiv;\Omega^-) \to P_0(\Omega^-)\,I$. 
\begin{lemma}\label{cota-tau}
There exists $c_1 > 0$, depending only on $\Omega^-$, such that
\begin{equation}\label{eq-cota-tau}
\|\bta^{\tt d}\|^2_{0,\Omega^-} \,+\,\|\bdiv \,\bta\|^2_{0,\Omega^-} 
\,\,\ge\,\,c_1\,\|\bta-\boldsymbol\pi_{_I} \bta\|^2_{0,\Omega^-}
\qquad\forall\,\bta\in \mathbb H(\bdiv;\Omega^-)\,.
\end{equation}
\end{lemma}
\begin{proof}
See \cite[Lemma 3.1]{adg} or \cite[Proposition 3.1, Chapter IV]{bf-1991}.

\end{proof}
\begin{lemma}\label{otra-cota-tau}
There exists $c_2 > 0$, depending only on $\Omega^-$, such that
\begin{equation}\label{eq-otra-cota-tau}
\|\bta-\boldsymbol\pi_{_I} \bta\|^2_{\bdiv;\Omega^-}\,\,\ge\,\,c_2\,\|\bta\|^2_{\bdiv;\Omega^-}
\qquad\forall\,\bta\,\in\,\mathbb H_0(\bdiv;\Omega^-)\,.
\end{equation}
\end{lemma}
\begin{proof}
See \cite[Lemma 4.5]{gmm2007}.

\end{proof}

\medskip
We are now in a position to establish the main result of this section.
\begin{theorem}
Given $\f\,\in\,\mathbf L^2(\Omega^-)$, there exists a unique 
$((\bsi,\bva),(\u,\bch))\,\in\,{\widetilde{\X}_N}\times{\Y_N}$ 
solution to \eqref{dual-mixed-formulation-complete-new}. In addition, there exists $C > 0$ such that
\[
\|((\bsi,\bva),(\u,\bch))\|_{{\X}_N\times {\Y_N}}\,\,\le\,\,C\,\|\f\|_{0,\Omega^-}\,.
\]
\end{theorem}
\begin{proof}
It reduces to verify the hypotheses of the classical Babu\v ska-Brezzi
theory.  The boundedness of $\a_N$ and $\b_N$ was already noticed at the
beginning of this section. Also, we observe that Lemma \ref{inf-sup-b}
establishes the required inf-sup condition for $\b_N$. Next, because of the
replacement of the space ${\X_N}$ by ${\widetilde{\X}_N}$ (cf.
\eqref{space-X-0}), the kernel of the operator induced by $\b_N : \widetilde\X_N \times \Y_N \to \R$ becomes now
\[
\widetilde{\mathbf V}_N\,:=\,\Big\{\,(\bta,\bps)\,\in\,\mathbb H_0(\bdiv;\Omega^-)\times\mathbf H^{1/2}_0(\Gamma): 
\quad \bta\,=\,\bta^{\tt t} \qan \bdiv\,\bta \,=\,\0 \qin \Omega^-\,\Big\}\,.
\]
Hence, applying \eqref{eq-quasi-coerciveness-a} (cf. Lemmas \ref{quasi-coerciveness-a}),
\eqref{eq-cota-tau} (cf. Lemma \ref{cota-tau}), \eqref{eq-otra-cota-tau} (cf. Lemma \ref{otra-cota-tau}),
and \eqref{W-eliptico} (cf. Lemma \ref{properties-2}),
we deduce that for any $(\bta,\bps)\,\in\,\widetilde\V_N$ there holds
\[
\begin{array}{l}
\disp
\a_N((\bta,\bps),(\bta,\bps))\,\ge\,
\frac12\,\|\bta^{\tt d}\|^2_{0,\Omega^-}\,+\,\frac12\,\langle \W\,\bps,\bps\rangle_\Gamma\,\ge\,
\frac{c_1}{2}\,\|\bta-\boldsymbol\pi_{_I}\bta\|^2_{0,\Omega^-}\,+\,\frac12\,\langle \W\,\bps,\bps\rangle_\Gamma\\[2ex]
\disp\quad
=\quad \frac{c_1}{2}\,\|\bta-\boldsymbol\pi_{_I}\bta\|^2_{\bdiv;\Omega^-}\,+\,\frac12\,\langle \W\,\bps,\bps\rangle_\Gamma\,\ge\,
\frac{c_1\,c_2}{2}\,\|\bta\|^2_{\bdiv;\Omega^-}\,+\,\frac{\alpha_2}{2}\,\|\bps\|^2_{1/2,\Gamma}\,,
\end{array}
\]
which proves that $\a_N$ is $\widetilde\V_N$-elliptic. In this way, 
the proof is completed by applying the corresponding result from the 
above mentioned theory (see, e.g. \cite[Theorem 1.1, Chapter II]{bf-1991}).

\end{proof}

\medskip
Notice that the result provided by the previous theorem constitutes the natural
extension of the continuous analysis developed in \cite{mss-MC-2011}, which in turn
adapts and modifies the main ideas from \cite{s-SINUM-2009}, to the present mixed
formulation of the three-dimensional exterior Stokes  problem. Furthermore, it is important
to remark at this point, as shown in the proof of Lemma
\ref{quasi-coerciveness-a},  that the $\widetilde\V_N$-ellipticity of $\a_N$, which is
certainly needed for the well-posedness of 
\eqref{dual-mixed-formulation-complete-new}, does require that the component $\bta$
of each pair  $(\bta,\bps)$ in $\widetilde\V_N\,$ be free-divergence and symmetric. In
particular, recall that the symmetry of $\bta$ is employed to replace
$\,\disp\int_{\Omega^-} \nabla \D\,\bps:\bta$ \,by\, $\disp\,\int_{\Omega^-}
\e(\D\,\bps):\bta$ \,in\, equation \eqref{above}, which constitutes a crucial
identity for the remaining part of the proof. Analogously, 
for the analysis of an  associated Galerkin scheme, one would need to
show that $\a_N$ is elliptic at the discrete kernel of $\b_N$, which is given by
\[
\begin{array}{c}
\disp
\widetilde\V_{N,h}\,:=\,\Bigg\{\,(\bta_h,\bps_h)\in\widetilde{\X}_{N,h}\,:=\,
\mathbb H^\bsi_{h,0} \times \mathbf H^\bva_{h,0}: \quad
\int_{\Omega^-} \bv_h \cdot\bdiv \bta_h\,=\,0 \quad \forall\,\bv_h\,\in\, \mathbf L^\u_h\\[2ex]
\disp\qan  \int_{\Omega^-} \bta_h : \bet_h\,=\,0 \quad \forall\,\bet_h\,\in\,\mathbb L^\bch_h    \,\Bigg\}\,,
\end{array}
\]
where $\mathbb H^\bsi_{h,0}$, $\mathbf H^\bva_{h,0}$, $\mathbf L^\u_h$, and $\mathbb
L^\bch_h$ are finite dimensional subspaces of $\mathbb H_0(\bdiv;\Omega^-)$, $\mathbf
H^{1/2}_0(\Gamma)$, $\mathbf L^2(\Omega^-)$, and $\mathbb L^2_{\tt skew}(\Omega^-)$,
respectively. Nevertheless, while it is possible to choose these subspaces so that
the discrete inf-sup condition for $\b_N$ is satisfied and the first equation
defining $\widetilde\V_{N,h}$ yields the components $\bta_h$ of the pairs $(\bta_h,\bps_h)\in
\widetilde\V_{N,h}$ to be free-divergence, no subspaces implying additionally the symmetry of
these components from the second equation defining $\widetilde\V_{N,h}$ are known (at least, up
to the authors' knowledge).  In order to overcome this difficulty, one could
consider Galerkin schemes for the simplified continuous formulation that arises from
\eqref{dual-mixed-formulation-complete-new} after eliminating the vorticity unknown
$\bch$, which means that one looks, from the beginning, for a symmetric stress tensor $\bsi$. The
recent availability of new stable mixed finite element methods for linear elasticity
with strong symmetry allows for the choice of concrete finite element subspaces 
towards this purpose (see, e.g. \cite{afw-acta-2006}, \cite{ac-JSC-2005}). However,
due to the high number of local degrees of freedom involved, this procedure
is still a bit prohibitive. Alternatively, instead of proving the $\widetilde\V_{N,h}$-ellipticity
of $\a_N$, one could try to show that this bilinear form satisfies the discrete inf-sup condition
on $\widetilde\V_{N,h}$, hoping that the symmetry property in question is not needed
along the way. However, this idea is rather an open question that needs to be 
further investigated. In the present paper we suggest a different approach
which makes no use of any strong symmetry property of the discrete tensors. 
More precisely, we show below in Section \ref{section41} that, under a suitable assumption 
on the mesh sizes involved, $\mathbf a_N$ does become uniformly strongly coercive 
on the discrete kernels of $\mathbf b_N$.

\medskip
On the other hand, another procedure that certainly avoids the need of any 
symmetry condition, neither for the continuous nor for the discrete kernels
of $\b_N$, is based on the incorporation of both integral equations \eqref{bie-1-with-bva} 
and \eqref{bie-2-with-bva} into the respective variational formulation.
This coupling method, known as the Costabel \& Han procedure, which has been denoted 
C \& H in Section \ref{section1}, is analyzed with Dirichlet and Neumann boundary 
conditions on $\Gamma_0$ in the forthcoming sections.

\subsection{C \& H coupling with Dirichlet boundary conditions on $\Gamma_0$}\label{section332}
We now follow the Costabel \& Han coupling method (see \cite{c1987}, \cite{h1990}) and incorporate the
boundary integral equations \eqref{bie-1-with-bva} and \eqref{bie-2-with-bva} into the dual-mixed variational 
formulation in $\Omega^-$ given by \eqref{dual-mixed-formulation-Dirichlet}, which assumes the Dirichlet
boundary condition $\,\gamma^-(\u)\,=\,\mathbf g_{_D}\,\in\,\mathbf H^{1/2}(\Gamma_0)$ \,on\, $\Gamma_0$.
More precisely, we replace $\bva$ in the first equation of \eqref{dual-mixed-formulation-Dirichlet} 
by the right hand side of \eqref{bie-1-with-bva}, and simultaneously
add \eqref{bie-2-with-bva} tested against $\bps\,\in\,\mathbf H^{1/2}(\Gamma)$ to the same equation, 
thus yielding a redefinition of the bilinear form $\a_D$ (cf. \eqref{a-D}). In this way, our coupled variational 
formulation reads as follows: Find $((\bsi,\bva),(\u,\bch))\,\in\,{\X_D}\times{\Y_D}$ such that
\begin{equation}\label{dual-mixed-formulation-Dirichlet-complete}
\begin{array}{rcl}
\mathbf a_D((\bsi,\bva),(\bta,\bps))\,+\,\mathbf b_D((\bta,\bps),(\u,\bch)) &=& \F_D(\bta,\bps)
\qquad\forall\,(\bta,\bps)\in \X_D\,,\\[2ex]
\mathbf b_D((\bsi,\bva),(\bv,\bet)) &=& \G_D(\bv,\bet)
\qquad\forall\,(\bv,\bet)\in \Y_D\,,
\end{array}
\end{equation}
where
\[
\mathbf X_D\,:=\,\mathbb H(\bdiv;\Omega^-)\times\mathbf H^{1/2}(\Gamma)\,,
\quad \mathbf Y_D\,:=\, \mathbf L^2(\Omega^-)\times \mathbb L^2_{\tt skew}(\Omega^-) \,,
\]
$\a_D : \X_D \times \X_D \to \R$ \,and\, $\b_D : \X_D \times \Y_D \to \R$ are the bounded bilinear forms
defined by
\begin{equation}\label{a-D-complete}
\begin{array}{c}
\disp
\mathbf a_D((\bsi,\bva),(\bta,\bps)) := \int_{\Omega^-} \bsi^{\tt d} : \bta^{\tt d}
\,+\,\langle \W\,\bva,\bps\rangle_\Gamma\,+\,\left\langle \left(\frac12\,\I\,+\,\K^{\tt t} \right)\,
\gamma^-_\bnu(\bsi),\bps \right\rangle_\Gamma \\[2ex]
\disp
+ \quad \langle \gamma^-_\bnu(\bta),\V \gamma^-_\bnu(\bsi)\rangle_\Gamma \,-\, \left\langle \gamma^-_\bnu(\bta) , \left(\frac12\,\I\,+\,\K \right)\bva \right\rangle_\Gamma  
\end{array}
\end{equation}
and
\begin{equation}\label{b-D-complete}
\mathbf b_D((\bta,\bps),(\bv,\bet))\,:=\,
\int_{\Omega^-} \bv \cdot \bdiv\,\bta\,+\,\int_{\Omega^-} \bet : \bta \,,
\end{equation}
and $\F_D : \X_D \to \R$ and $\G_D : \Y_D \to \R$ are the bounded linear functionals given by
\begin{equation}\label{F-D-complete}
\F_D(\bta,\bps)\,:=\,-\,\langle \gamma^-_\bnu(\bta),\g_{_D}\rangle_{\Gamma_0}\,
\qquad
\mbox{and}
\qquad
\G_D(\bv,\bet)\,:=\,-\,\int_{\Omega^-}\f\cdot\bv \,.
\end{equation}

\medskip
We first let $\V_D$ be the kernel of the bounded linear operator induced by $\mathbf b_D$ 
(cf. \eqref{b-D-complete}), that is 
\[
\mathbf V_D\,:=\,\Big\{\,(\bta,\bps)\,\in\,{\X_D}: \quad \bta\,=\,\bta^{\tt t} \qan \bdiv\,\bta \,=\,\0 
\qin \Omega^-\,\Big\}\,,
\]
and identify the solutions of the homogeneous problem associated with \eqref{dual-mixed-formulation-Dirichlet-complete}.
\begin{lemma}\label{kernel-Dirichlet}
The set of solutions of the homogeneous version of \eqref{dual-mixed-formulation-Dirichlet-complete}
is given by
\[
\Big\{\,\big((\bsi,\bva),(\u,\bch)\big)\,:=\,
\big((\0,\z),(\0,\0)\big): \qquad \z\,\in\,\mathbf{RM}(\Gamma)\,\Big\}\,.
\]
\end{lemma}
\begin{proof}
Let $\big((\bsi,\bva),(\u,\bch)\big)\,\in\,{\X_D} \times {\Y_D}$ be a solution of 
\eqref{dual-mixed-formulation-Dirichlet-complete} with $\g_{_D} = \0$ on $\Gamma_0$
and $\f\,=\,\0$ in $\Omega^-$. It is clear from the second equation that $(\bsi,\bva)\,\in\,\mathbf V_D$,
\,that is\, $\bsi\,=\,\bsi^{\tt t}$ \,and\, $\bdiv \bsi\,=\,\0$ \,in\, $\Omega^-$.
Then, taking in particular $(\bta,\bps)\,=\,(\bsi,\bva)$ in the first equation,
recalling that $\K^{\tt t}$ is the adjoint of $\K$,
and then applying the inequalities \eqref{eq-cota-tau} (cf. Lemma 
\ref{cota-tau}), \eqref{W-eliptico-equiv}, and \eqref{V-eliptico-equiv}, we find that
\begin{equation}\label{eq-a-D-0}
\begin{array}{rl}
\disp
0 &\,=\,\mathbf a_D\big((\bsi,\bva),(\bsi,\bva)\big)\,=\,
\|\bsi^{\tt d}\|^2_{0,\Omega^-}\,+\,\langle \W\,\bva,\bva\rangle_\Gamma\,+\,
\langle\gamma^-_\bnu(\bsi),\V \gamma^-_\bnu(\bsi)\rangle_\Gamma\\[2ex]
&
\disp
\ge \quad c_1\|\bsi-\boldsymbol\pi_{_I}\bsi\|^2_{0,\Omega^-}\,
+\,\tilde\alpha_2\,\|\bva-\boldsymbol\pi_{_{RM}}\bva\|^2_{1/2,\Gamma}\,+\,
\tilde\alpha_1\,\|\gamma^-_\bnu(\bsi)-\boldsymbol\pi_{_\nu} \gamma^-_\bnu(\bsi)\|^2_{-1/2,\Gamma}\,.
\end{array}
\end{equation}
Therefore $\bsi\,=\,c\,\I$ \,and\, $\bva\,=\,\mathbf z\in \mathbf{RM}(\Gamma)$.
As a consequence, and using the characterization of the kernels of $\V$ and $\W$ 
given by \eqref{kernels-V-and-W}, we find that the first equation of the homogeneous 
\eqref{dual-mixed-formulation-Dirichlet-complete} becomes 
\[
c\,\langle \bnu,\bps\rangle_\Gamma\,+\,\b_D((\bta,\bps),(\u,\bch))\,=\,0\qquad\forall\,(\bta,\bps)\,\in\,\X_D\,.
\]
Then, taking $\bps\,=\,\0$ in the above equation, and using the inf-sup condition \eqref{eq-inf-sup-b-simple},
which is possible in this case thanks to the inclusion $\mathbb H_0(\bdiv;\Omega^-)\,\subseteq\,\mathbb H(\bdiv;\Omega^-)$
and the fact that the expressions defining $\b_N$ and $\b_D$ coincide, we deduce that $(\u,\bch)\,=\,(\0,\0)$.
In this way, we obtain that for any $\bps\,\in\,\mathbf H^{1/2}(\Gamma)$ there holds 
$\,c\,\langle \bnu,\bps\rangle_\Gamma\,=\,0$, which necessarily implies that $c\,=\,0$, and thus $\,\bsi\,=\,\0$.
Conversely, it is not difficult to see, using again the characterization of $\mathrm{ker}\,(\W)$ 
(cf. \eqref{kernels-V-and-W}), that for any $\z\in \mathbf{RM}(\Gamma)$, 
$\,\big((\bsi,\bva),(\u,\bch)\big)\,:=\,\big((\0,\z),(\0,\0)\big)$ solves the 
homogeneous version of \eqref{dual-mixed-formulation-Dirichlet-complete}.

\end{proof}

\medskip
Similarly as for the analysis in Section \ref{section331},
and in order to guarantee the unique solvability of the coupled problem \eqref{dual-mixed-formulation-Dirichlet-complete}, 
we now look for the solution $((\bsi,\bva),(\u,\bch))$ in the space ${\widetilde{\X}_D}\times{\Y_D}$, where
\begin{equation}\label{tilde-X-D}
\widetilde{\X}_D\,:=\,\mathbb H(\bdiv;\Omega^-)\times\mathbf H^{1/2}_0(\Gamma)\,,
\end{equation}
which yields the kernel of the operator defined by $\b_D : \widetilde\X_D \times \Y_D \to \R$ to become
\[
\widetilde\V_D\,:=\,\Big\{\,(\bta,\bps)\,\in\,\widetilde \X_D : 
\quad \bta\,=\,\bta^{\tt t} \qan \bdiv\,\bta \,=\,\0 \qin \Omega^-\,\Big\}\,.
\]
In turn, thanks again to the characterization of the kernels (cf. \eqref{kernels-V-and-W}) and the 
symmetry-type property of $\W$, we deduce that
it suffices to require the first equation of \eqref{dual-mixed-formulation-Dirichlet-complete} 
for any $(\bta,\bps) \in{\widetilde{\X}_D}$. Therefore, instead of \eqref{dual-mixed-formulation-Dirichlet-complete},
we now look for $((\bsi,\bva),(\u,\bch))\,\in\,{\widetilde{\X}_D}\times{\Y_D}$ such that
\begin{equation}\label{dual-mixed-formulation-Dirichlet-complete-new}
\begin{array}{rcl}
\mathbf a_D((\bsi,\bva),(\bta,\bps))\,+\,\mathbf b_D((\bta,\bps),(\u,\bch)) &=& \F_D(\bta,\bps)
\qquad\forall\,(\bta,\bps)\,\in\, {\widetilde{\X}_D}\,,\\[2ex]
\mathbf b_D((\bsi,\bva),(\bv,\bet)) &=& \G_D(\bv,\bet)
\qquad\forall\,(\bv,\bet)\,\in\, {\Y_D}\,,
\end{array}
\end{equation}

\medskip
Note, according to the definition of $\a_D$ (cf. \eqref{a-D-complete}) and the identity
$\V(\bnu)\,=\,\0$ (cf. \eqref{kernels-V-and-W}), that for any $c\in\R$ there holds
$\,\a_D((c\,\I,\0),(c\,\I,\0))\,=\,0$, which proves that $\a_D$ is not $\widetilde\V_D$-elliptic.
However, the following lemma establishes the weak-coerciveness of $\a_D$ on this kernel. 
\begin{lemma}\label{weak-coercive-a-D}
The bilinear form $\a_D:\widetilde\V_D\times \widetilde\V_D \to \R$ defines an invertible operator 
$\mathbf A_D:\widetilde \V_D\to \widetilde\V_D$.
\end{lemma}

\begin{proof}
Using that $\K^{\tt t}$ is the adjoint of $\K$ and the inequalities \eqref{eq-cota-tau}, \eqref{W-eliptico}, and \eqref{V-eliptico-equiv},
it is easy to show that
\begin{equation}
\mathbf a_D\big((\bsi,\bva),(\bsi,\bva)\big)+\| \boldsymbol\pi_{_I}\bsi\|_{0,\Omega^-}^2 \,\ge\,
C\,\Big\{\|\bsi\|^2_{\bdiv;\Omega^-}\,+\,\|\bva\|^2_{1/2,\Gamma}\Big\}
\qquad\forall\,(\bsi,\bva)\,\in\,\widetilde\V_D.
\end{equation}
Therefore, the operator $\mathbf A_D$ is Fredholm of index zero. If $\mathbf A_D((\bsi,\bva))=0$, then (by the same arguments)
\begin{equation}
0=\mathbf a_D\big((\bsi,\bva),(\bsi,\bva)\big)\ge C\,
\Big\{ \|\bsi-\boldsymbol\pi_{_I}\bsi\|^2_{\bdiv;\Omega^-}\,+\,\|\bva\|^2_{1/2,\Gamma}\Big\},
\end{equation}
which implies that $\bva=0$ and $\bsi=c\I$ for some $c\in \R$. Therefore $\bsi^{\tt d}=0$, 
$\gamma^-_\bnu(\bsi)=c\bnu$ and, using \eqref{kernels-V-and-W}
\[
0=\mathbf a_D\big((\bsi,\bva),(\0,\bps)\big)=\,c\left\langle \left(\frac12\,\I\,+\,\K^{\tt t} \right)\,
\bnu,\bps \right\rangle_\Gamma =\langle\bnu,\bps\rangle_\Gamma \qquad \forall \bps\in \H^{1/2}_0(\Gamma).
\]
This identity, \eqref{new4.24}, and \eqref{decompositions} imply that $c=0$. We have thus proved that $\mathbf A_D$ is injective, and therefore invertible.
\end{proof}

\medskip
The well-posedness of \eqref{dual-mixed-formulation-Dirichlet-complete-new} can now be established.
\begin{theorem}
Given $\g_{_D}\,\in\,\mathbf H^{1/2}(\Gamma_0)$ and $\f\,\in\,\mathbf L^2(\Omega^-)$, there exists a unique 
$((\bsi,\bva),(\u,\bch))\,\in\,{\widetilde{\X}_D}\times{\Y_D}$ 
solution to \eqref{dual-mixed-formulation-Dirichlet-complete-new}. In addition, there exists $C > 0$ such that
\[
\|((\bsi,\bva),(\u,\bch))\|_{{\X}_D\times {\Y_D}}\,\,\le\,\,
C\,\Big\{\,\|\g_{_D}\|_{1/2,\Gamma_0}\,+\,\|\f\|_{0,\Omega^-}\,\Big\}\,.
\]
\end{theorem}
\begin{proof}
It is clear from the beginning of the section that $\a_D$ and $\b_D$ are bounded bilinear forms.
In addition, the continuous inf-sup condition for $\b_D : \widetilde\X_D \times \Y_D \to \R$
follows straightforwardly from \eqref{eq-inf-sup-b-simple} by noting that the expressions
defining $\b_D$ (cf. \eqref{b-D-complete}) and $\b_N$ (cf. \eqref{b-N-complete}) 
coincide and that certainly $\mathbb H_0(\bdiv;\Omega^-)
\,\subseteq\,\mathbb H(\bdiv;\Omega^-)$. Consequently, the proof is completed by applying the 
corresponding result from the Babu\v ska-Brezzi theory 
(see, e.g. \cite[Theorem 1.1, Chapter II]{bf-1991}).

\end{proof}

\medskip
The proof of weak coerciveness of $\a_D$ (cf. Lemma \ref{weak-coercive-a-D}) uses a compactness argument starting from a G\r arding inequality. At the discrete level this would impose a certain restriction on the approximation properties of the space to be fine enough. Note that an even more indirect argument (by contradiction) was used in \cite[Lemma 4.3]{bcs1996} for the analysis of the 
coupling of mixed-FEM and BEM as applied to the elasticity problem. However, a mesh-dependent norm 
had to be employed there instead of the usual norm of $\mathbb H(\bdiv;\Omega^-)$, and it is not clear 
from the proof whether the constants involved depend or not on $h$. In the next Lemma we are going to give a refined version of the weak-coercivity that can be inherited at the discrete level.

\begin{lemma}\label{weak-coercive-a-D-direct}
Let $\widetilde\bps$ be an arbitrary but fixed element in $\mathbf H^{1/2}(\Gamma)$ such that
$\langle \bnu,\widetilde\bps\rangle_\Gamma \,>\,0$. Consider the functional $c:\mathbb H(\bdiv,\Omega^-)\to \R$ given by
\[
c(\bsi):=\frac1{|\Omega^-|} \int_{\Omega^-} \tr\bsi.
\]
Then there exist positive $C$ and $\delta$, depending on $\widetilde\bps$, such that
\[
\a_D\big((\bsi,\bva),(\bsi,\bva+\delta c(\bsi) \widetilde\bps)\big)\ge C \|(\bsi,\bva)\|_{\X_D}^2 \qquad
\forall (\bsi,\bva)\in \widetilde\V_D.
\]
\end{lemma}

\begin{proof}
Because of \eqref{new4.24} we can assume that $\widetilde\bps\in \H^{1/2}_0(\Gamma)$ (just write $\widetilde\bps-\boldsymbol\pi_{_{RM}}
\widetilde\bps$ instead of $\widetilde\bps$).
Let us next notice that $\boldsymbol\pi_{_I}\bsi=c(\bsi)\I$ and therefore
\[
\a_D\big(\boldsymbol\pi_{_I}\bsi,\0),(\0,c(\bsi)\widetilde\bps)\big)=c(\bsi)^2 \langle\bnu,\widetilde\bps\rangle_\Gamma.
\]
This inequality, combined with the bound
\[
|\a_D\big((\bsi-\boldsymbol\pi_{_I}\bsi,\bva),(0,c(\bsi)\widetilde\bps)\big)|
\le
\frac12 c(\bsi)^2 \langle\bnu,\widetilde\bps\rangle_\Gamma
+ 
\frac12 \frac{\|\a_D\|^2\|\widetilde\bps\|_{1/2,\Gamma}^2}{ \langle\bnu,\widetilde\bps\rangle_\Gamma}
\| \bsi-\boldsymbol\pi_{_I}\bsi\|_{\X_D}^2
\]
prove
\begin{equation}\label{eq:5.24}
\a_D\big((\bsi,\bva),(0,c(\bsi)\widetilde\bps)\big)\ge 
\frac12 c(\bsi)^2 \langle\bnu,\widetilde\bps\rangle_\Gamma
- \frac12 \frac{\|\a_D\|^2\|\widetilde\bps\|_{1/2,\Gamma}^2}{ \langle\bnu,\widetilde\bps\rangle_\Gamma}
\| \bsi-\boldsymbol\pi_{_I}\bsi\|_{\X_D}^2.
\end{equation}
At the same time, we know (see the proof of Lemma  \ref{weak-coercive-a-D}) that there exists $\widetilde C>0$ such that
\begin{equation}\label{eq:5.25}
\a_D\big((\bsi,\bva),(\bsi,\bva)\big) \ge \widetilde C \| (\bsi-\boldsymbol\pi_{_I}\bsi,\bva)\|_{\X_D}^2 \qquad
\forall (\bsi,\bva)\in \widetilde\V_D.
\end{equation}
Taking
\[
\delta:=\frac{\widetilde C \langle\bnu,\widetilde\bps\rangle_\Gamma}{\|\a_D\|^2\| \widetilde\bps\|_{1/2,\Gamma}^2}
\]
and combining \eqref{eq:5.24}-\eqref{eq:5.25}, we easily prove that
\begin{equation}
\a_D\big((\bsi,\bva),(\bsi,\bva+\delta c(\bsi)\widetilde\bps)\big)
\ge 
\frac{\widetilde C \langle\bnu,\widetilde\bps\rangle_\Gamma^2}{2 \|\a_D\|^2 \|\widetilde\bps\|_{1/2,\Gamma}^2}
c(\bsi)^2 +\frac{\widetilde C}2 \|(\bsi-\boldsymbol\pi_{_I} \bsi,\bva)\|_{\X_D}^2 \,.
\end{equation}
This finishes the proof.
\end{proof}

\medskip
We note that the result of Lemma \ref{weak-coercive-a-D-direct} fits into what is called a T-coercivity result. In this case, the bounded isomorphism $T(\bsi,\bva):=(\bsi,\bva+c(\bsi) \delta\widetilde\bps)$ applied to the second component of the bilinear form makes it coercive. This gives an alternative proof of Lemma \ref{weak-coercive-a-D} providing additionally an estimate of the norm of the inverse of the operator $\mathbf A_D$, which depends exclusively on $\langle\bnu,\widetilde\bps\rangle_\Gamma$, $\|\widetilde\bps-\boldsymbol\pi_{_{RM}}\widetilde\bps\|_{1/2,\Gamma}$, the norm of $\a_D$, and the coercivity constant of \eqref{eq:5.25}.

\subsection{C \& H coupling with non-homogeneous Neumann boundary conditions on $\Gamma_0$}\label{section333}
Similarly as in the previous section, we now apply again the Costabel \& Han coupling method (see \cite{c1987}, \cite{h1990}) 
and incorporate the boundary integral equations \eqref{bie-1-with-bva} and \eqref{bie-2-with-bva} into the dual-mixed variational 
formulation in $\Omega^-$ given by \eqref{dual-mixed-formulation-Neumann}, which considers the 
non-homogeneous Neumann boundary condition 
$\,\gamma_\bnu^-(\bsi)\,=\,\mathbf g_{_N}\,\in\,\mathbf H^{-1/2}(\Gamma_0)$ \,on\, $\Gamma_0$.
In this way, our coupled variational formulation reads as follows: 
Find $((\bsi,\bva),(\u,\bch,\bla))\,\in\,{\X_N}\times{\Y_N}$ such that
\begin{equation}\label{dual-mixed-formulation-Neumann-complete}
\begin{array}{rcl}
\mathbf a_N((\bsi,\bva),(\bta,\bps))\,+\,\mathbf b_N((\bta,\bps),(\u,\bch,\bla)) &=& \F_N(\bta,\bps)
\qquad\forall\,(\bta,\bps)\in \X_N\,,\\[2ex]
\mathbf b_N((\bsi,\bva),(\bv,\bet,\bxi)) &=& \G_N(\bv,\bet,\bxi)
\qquad\forall\,(\bv,\bet,\bxi)\in \Y_N\,,
\end{array}
\end{equation}
where
\[
\mathbf X_N\,:=\,\mathbb H(\bdiv;\Omega^-)\times\mathbf H^{1/2}(\Gamma)\,,
\quad \mathbf Y_N\,:=\, \mathbf L^2(\Omega^-)\times \mathbb L^2_{\tt skew}(\Omega^-)
\times \mathbf H^{1/2}(\Gamma_0) \,,
\]
$\a_N : \X_N \times \X_N \to \R$ \,and\, $\b_N : \X_N \times \Y_N \to \R$ are the bounded bilinear forms
defined by
\begin{equation}\label{a-N-complete-N}
\begin{array}{c}
\disp
\mathbf a_N((\bsi,\bva),(\bta,\bps)) := \int_{\Omega^-} \bsi^{\tt d} : \bta^{\tt d}
\,+\,\langle \W\,\bva,\bps\rangle_\Gamma\,+\,\left\langle \left(\frac12\,\I\,+\,\K^{\tt t} \right)\,
\gamma^-_\bnu(\bsi),\bps \right\rangle_\Gamma \\[2ex]
\disp
+ \quad \langle \gamma^-_\bnu(\bta),\V \gamma^-_\bnu(\bsi)\rangle_\Gamma \,-\, \left\langle \gamma^-_\bnu(\bta) , \left(\frac12\,\I\,+\,\K \right)\bva \right\rangle_\Gamma  
\end{array}
\end{equation}
and
\begin{equation}\label{b-N-complete-N}
\mathbf b_N((\bta,\bps),(\bv,\bet,\bxi))\,:=\,
\int_{\Omega^-} \bv \cdot \bdiv\,\bta\,+\,\int_{\Omega^-} \bet : \bta \,+\, 
\langle \gamma^-_\bnu(\bta),\bxi\rangle_{\Gamma_0}\,,
\end{equation}
and $\F_N : \X_N \to \R$ and $\G_N : \Y_N \to \R$ are the bounded linear functionals given by
\begin{equation}\label{F-N-complete-N}
\F_N(\bta,\bps)\,:=\,0 \,,
\qquad
\mbox{and}
\qquad
\G_N(\bv,\bet,\bxi)\,:=\,-\,\int_{\Omega^-}\f\cdot\bv
\,+\, \langle \g_{_N},\bxi\rangle_{\Gamma_0} \,.
\end{equation}

\medskip
We now observe from \eqref{b-N-complete-N} that $\mathbf B_N : \X_N \to \Y_N$, 
the bounded linear operator induced by $\b_N$, is given by  
\[
\mathbf B_N((\bta,\bps))\,:=\,\big(
\bdiv\bta,\frac12(\bta - \bta^{\tt t}),\mathcal R_0\,\gamma_\bnu^-(\bta)\big) \qquad
\forall\,(\bta,\bps)\,\in\,\X_N\,,
\]
where $\,\mathcal R_0 : \mathbf H^{-1/2}(\Gamma_0) \to \mathbf H^{1/2}(\Gamma_0)$
is the respective Riesz operator. Then, we begin the analysis of solvability of
\eqref{dual-mixed-formulation-Neumann-complete} by proving next that $\b_N$ satisfies
the continuous inf-sup condition, which is equivalent to the surjectivity of $\mathbf B_N$.
\begin{lemma}\label{inf-sup-b-N-complete}
There exists $\beta > 0$ such that for any $(\bv,\bet,\bxi)\,\in\,{\Y_N}$ there holds
\begin{equation}\label{eq-inf-sup-b-N-complete}
\sup_{(\bta,\bps)\in {\X_N}\backslash\{\0\}}\,
\frac{\mathbf b_N((\bta,\bps),(\bv,\bet,\bxi))}{\|(\bta,\bps)\|_{\X_N}}\,\ge\,\beta\,\|(\bv,\bet,\bxi)\|_{\Y_N}\,.
\end{equation}
\end{lemma}
\begin{proof}
We proceed as in the proof of Lemma \ref{inf-sup-b}. In fact, given
$(\bv,\bet,\bxi)\,\in\,{\Y_N}$, we let $\z$ be the unique element in $\,\mathbf H^1_\Gamma(\Omega^-)
\,:=\,\Big\{\,\mathbf w\in \mathbf H^1(\Omega^-): \quad \mathbf w\,=\,\0 \qon \Gamma\,\Big\}\,$,
whose existence is guaranteed by the second Korn inequality, such that
\[
\int_{\Omega^-} \e(\z) : \e(\w)\,=\,-\int_{\Omega^-} \bv \cdot \w
\,- \int_{\Omega^-} \bet : \nabla \w \,+\,\langle \mathcal R^{-1}_0(\bxi),\gamma^-(\w)\rangle_{\Gamma_0}
\qquad\forall\,\w\,\in\,\mathbf H^1_\Gamma(\Omega^-)\,.
\]
Hence, defining $\,\,\widehat\bta\,:=\,\e(\z) + \bet\,\in\,\mathbb L^2(\Omega^-)$,
we deduce from the above formulation that $\,\bdiv\,\widehat\bta\,=\,\bv$ \,in\, $\Omega^-$, which shows that
$\,\widehat\bta\,\in\,\mathbb H(\bdiv;\Omega^-)$, and then that $\gamma^-_\bnu(\widehat\bta)
\,=\,\mathcal R^{-1}_0(\bxi)$ \,on\, $\Gamma_0$. In this way, it is easy to see that 
$\,\mathbf B_N((\widehat\bta,\0))\,=\,(\bv,\bet,\bxi)$, which ends the proof.

\end{proof}

\medskip
In what follows we let $\V_N$ be the kernel of $\mathbf B_N$, that is
\[
\V_N\,:=\,\Big\{ \,(\bta,\bps)\in \X_N: \quad \bta = \bta^{\tt t} \qan \bdiv \bta \,=\,\0\qin
\Omega^-\,, \qan \gamma^-_\bnu(\bta)\,=\,\0 \qon \Gamma_0\,\Big\}\,,
\]
and establish a positiveness property of $\a_N$ on $\V_N$.
\begin{lemma}\label{quasi-coerciveness-a-N}
There exists $\widetilde\alpha > 0$ such that
\begin{equation}\label{eq-quasi-coerciveness-a-N}
\a_N((\bta,\bps),(\bta,\bps))\,\ge\,\widetilde\alpha\,\Big\{\,
\|\bta\|^2_{\bdiv;\Omega^-}\,+\,\|\bps-\boldsymbol\pi_{_{RM}}\bps\|^2_{1/2,\Gamma}\,\Big\}
\qquad\forall\,(\bta,\bps)\,\in\,\V_N.,
\end{equation}
\end{lemma}
\begin{proof}
Let $(\bta,\bps)\in \V_N$. Then, recalling that $\K^{\tt t}$ is the adjoint of $\K$, noting that 
$\bta\in\mathbb H_0(\bdiv;\Omega^-)$, and then applying the inequalities \eqref{eq-cota-tau} (cf. Lemma 
\ref{cota-tau}), \eqref{eq-otra-cota-tau} (cf. Lemma \ref{otra-cota-tau}),
\eqref{W-eliptico-equiv}, and \eqref{V-eliptico-equiv}, we find that
\begin{equation}\label{eq-a-N-0}
\begin{array}{l}
\disp
\mathbf a_N\big((\bta,\bps),(\bta,\bps)\big)\,=\,
\|\bta^{\tt d}\|^2_{0,\Omega^-}\,+\,\langle \W\,\bps,\bps\rangle_\Gamma\,+\,
\langle\gamma^-_\bnu(\bta),\V \gamma^-_\bnu(\bta)\rangle_\Gamma\\[2ex]
\disp\qquad
\ge \quad c_1\|\bta-\boldsymbol\pi_{_I}\bta\|^2_{\bdiv;\Omega^-}\,+\,\tilde\alpha_2\,\|\bps-\boldsymbol\pi_{_{RM}}\bps\|^2_{1/2,\Gamma}\,+\,
\tilde\alpha_1\,\|\gamma^-_\bnu(\bta)-\boldsymbol\pi_{_\nu}\gamma^-_\bnu(\bta)\|^2_{-1/2,\Gamma}\\[2ex]
\disp\qquad
\ge \quad c_1\,c_2\,\|\bta\|^2_{\bdiv;\Omega^-}\,+\,\tilde\alpha_2\,\|\bps-\boldsymbol\pi_{_{RM}}\|^2_{1/2,\Gamma}\,,
\end{array}
\end{equation}
which yields the required inequality \eqref{eq-quasi-coerciveness-a-N}.

\end{proof}

\begin{lemma}\label{kernel-Neumann}
The set of solutions of the homogeneous version of \eqref{dual-mixed-formulation-Neumann-complete}
is given by
\[
\Big\{\,\big((\bsi,\bva),(\u,\bch,\bla)\big)\,:=\,
\big((\0,\z),(\0,\0,\0)\big): \qquad \z\,\in\,\mathbf{RM}(\Gamma)\,\Big\}\,.
\]
\end{lemma}
\begin{proof}
It follows similarly to the proof of Lemma \ref{kernel-Dirichlet}. Indeed, 
let $\big((\bsi,\bva),(\u,\bch,\bla)\big)\,\in\,{\X_N} \times {\Y_N}$ be a solution of 
\eqref{dual-mixed-formulation-Dirichlet-complete} with $\g_{_N} = \0$ on $\Gamma_0$
and $\f\,=\,\0$ in $\Omega^-$. It is clear from the second equation that $(\bsi,\bva)\,\in\,\mathbf V_N$.
Then, taking in particular $(\bta,\bps)\,=\,(\bsi,\bva)$ in the first equation,
and using the inequality \eqref{eq-quasi-coerciveness-a-N} (cf. Lemma \ref{quasi-coerciveness-a-N}), 
we find that 
\begin{equation}\label{eq-a-N-0-otra}
0\,=\,\mathbf a_N\big((\bsi,\bva),(\bsi,\bva)\big)\,\ge\,
\widetilde\alpha\,\Big\{\,\|\bsi\|^2_{\bdiv;\Omega^-}\,+\,\|\bva-\boldsymbol\pi_{_{RM}}\bva\|^2_{1/2,\Gamma}\,\Big\}\,,
\end{equation}
from where it follows that $\,\bsi\,=\,\0$ \,in\, $\Omega^-$ 
\,and\,  $\bva\,=\,\mathbf z$ for $\z\in \mathbf{RM}(\Gamma)$.
As a consequence, and using the characterization of the kernel $\W$ 
given by \eqref{kernels-V-and-W}, we find that the first equation of the homogeneous 
\eqref{dual-mixed-formulation-Neumann-complete} becomes 
\[
\b_N((\bta,\bps),(\u,\bch,\bla))\,=\,0\qquad\forall\,(\bta,\bps)\,\in\,\X_N\,,
\]
which, thanks to the inf-sup condition \eqref{eq-inf-sup-b-N-complete} (cf. Lemma \ref{inf-sup-b-N-complete}), yields
$(\u,\bch,\bla)\,=\,(\0,\0,\0)$.
Conversely, it is not difficult to see, using again the characterization of $\mathrm{ker}\,(\W)$ 
(cf. \eqref{kernels-V-and-W}), that for any $\z\in \mathbf{RM}(\Gamma)$, 
$\,\big((\bsi,\bva),(\u,\bch,\bla)\big)\,:=\,\big((\0,\z),(\0,\0,\0)\big)$ solves the 
homogeneous version of \eqref{dual-mixed-formulation-Neumann-complete}.

\end{proof}

\medskip
Therefore, similarly as for the analysis in Sections \ref{section331} and \ref{section332},
and in order to guarantee the unique solvability of the coupled problem 
\eqref{dual-mixed-formulation-Neumann-complete}, we now look for the solution
$((\bsi,\bva),(\u,\bch,\bla))$ in the space $\widetilde\X_N \times \Y_N$, where
\begin{equation}\label{tilde-X-N}
\widetilde{\X}_N\,:=\,\mathbb H(\bdiv;\Omega^-)\times\mathbf H^{1/2}_0(\Gamma)\,,
\end{equation}
which yields the kernel of the operator defined by $\b_N : \widetilde\X_N \times \Y_N \to \R$ to become
\[
\widetilde\V_N\,:=\,\Big\{ \,(\bta,\bps)\in \widetilde\X_N: \quad \bta = \bta^{\tt t} \qan \bdiv \bta \,=\,\0\qin
\Omega^-\,, \qan \gamma^-_\bnu(\bta)\,=\,\0 \qon \Gamma_0\,\Big\}\,.
\]
In addition, applying again the characterization of the kernels (cf. \eqref{kernels-V-and-W}) and the 
symmetry-type property of $\W$, we find that
it suffices to require the first equation of \eqref{dual-mixed-formulation-Neumann-complete} 
for any $(\bta,\bps) \in{\widetilde{\X}_N}$. As a consequence, instead of 
\eqref{dual-mixed-formulation-Neumann-complete}, we now look for 
$((\bsi,\bva),(\u,\bch,\bla))\,\in\,{\widetilde\X_N}\times{\Y_N}$ such that
\begin{equation}\label{dual-mixed-formulation-Neumann-complete-new}
\begin{array}{rcl}
\mathbf a_N((\bsi,\bva),(\bta,\bps))\,+\,\mathbf b_N((\bta,\bps),(\u,\bch,\bla)) &=& \F_N(\bta,\bps)
\qquad\forall\,(\bta,\bps)\in \widetilde\X_N\,,\\[2ex]
\mathbf b_N((\bsi,\bva),(\bv,\bet,\bxi)) &=& \G_N(\bv,\bet,\bxi)
\qquad\forall\,(\bv,\bet,\bxi)\in \Y_N\,,
\end{array}
\end{equation}

\medskip
Next, it follows from Lemma \ref{quasi-coerciveness-a-N} and the above 
characterization of $\widetilde\V_N$ that the bilinear form $\mathbf a_N$ is strongly coercive on
$\widetilde\V_N$. In addition, since actually $\mathbf b_N$ does not depend on the component
$\bps\in \mathbf H^{1/2}(\Gamma)$, it is quite clear from Lemma \ref{inf-sup-b-N-complete} 
that $\mathbf b_N$ satisfies the continuous inf-sup condition on $\widetilde \X_N \times \Y_N$ as well.
Hence, the well-posedness of \eqref{dual-mixed-formulation-Neumann-complete-new} is readily
established as follows.
\begin{theorem}\label{CH-nhNbc}
Given $\mathbf g_{_N}\,\in\,\mathbf H^{-1/2}(\Gamma_0)$ and $\mathbf f\,\in\,\mathbf L^2(\Omega^-)$, there
exists a unique $((\bsi,\bva),(\u,\bch,\bla))\,\in\,{\widetilde\X_N}\times{\Y_N}$ solution to
\eqref{dual-mixed-formulation-Neumann-complete-new}. In addition, there exists $C > 0$ such that
\[
\|((\bsi,\bva),(\u,\bch,\bla))\|_{\X_N \times \Y_N}\,\le\,C\,
\Big\{\,\|\mathbf g_{_N}\|_{-1/2,\Gamma_0}\,+\,\|\mathbf f\|_{0,\Omega^-}\,\Big\}\,.
\]
\end{theorem}
\begin{proof}
According to the previous discussion, the proof follows by applying 
once again the usual result from the Babu\v ska-Brezzi theory 
(see, e.g. \cite[Theorem 1.1, Chapter II]{bf-1991}).

\end{proof}
\subsection{C \& H coupling with homogeneous Neumann boundary conditions on $\Gamma_0$}\label{section334}
In what follows we proceed similarly as in the previous section and apply the Costabel \& Han coupling method 
to the case of homogeneous Neumann boundary conditions on $\Gamma_0$. This means that we now
incorporate the boundary integral equations \eqref{bie-1-with-bva} and \eqref{bie-2-with-bva} 
into the dual-mixed variational formulation in $\Omega^-$ given by \eqref{dual-mixed-formulation-Neumann-0}. 
In this way, as in Sections \ref{section313} and \ref{section331}, there is no need of introducing 
the additional unknown $\lambda \in \mathbf H^{1/2}(\Gamma_0)$, 
and hence our coupled variational formulation simply reads as follows: 
Find $((\bsi,\bva),(\u,\bch))\,\in\,{\X_N}\times{\Y_N}$ such that
\begin{equation}\label{dual-mixed-formulation-Neumann-complete-0}
\begin{array}{rcl}
\mathbf a_N((\bsi,\bva),(\bta,\bps))\,+\,\mathbf b_N((\bta,\bps),(\u,\bch)) &=& \F_N(\bta,\bps)
\qquad\forall\,(\bta,\bps)\in \X_N\,,\\[2ex]
\mathbf b_N((\bsi,\bva),(\bv,\bet)) &=& \G_N(\bv,\bet)
\qquad\forall\,(\bv,\bet)\in \Y_N\,,
\end{array}
\end{equation}
where
\[
\mathbf X_N\,:=\,\mathbb H_0(\bdiv;\Omega^-)\times\mathbf H^{1/2}(\Gamma)\,,
\quad \mathbf Y_N\,:=\, \mathbf L^2(\Omega^-)\times \mathbb L^2_{\tt skew}(\Omega^-) \,,
\]
$\a_N : \X_N \times \X_N \to \R$ \,and\, $\b_N : \X_N \times \Y_N \to \R$ are the bounded bilinear forms
defined by
\begin{equation}\label{a-N-complete-N-0}
\begin{array}{c}
\disp
\mathbf a_N((\bsi,\bva),(\bta,\bps)) := \int_{\Omega^-} \bsi^{\tt d} : \bta^{\tt d}
\,+\,\langle \W\,\bva,\bps\rangle_\Gamma\,+\,\left\langle \left(\frac12\,\I\,+\,\K^{\tt t} \right)\,
\gamma^-_\bnu(\bsi),\bps \right\rangle_\Gamma \\[2ex]
\disp
+ \quad \langle \gamma^-_\bnu(\bta),\V \gamma^-_\bnu(\bsi)\rangle_\Gamma \,-\, \left\langle \gamma^-_\bnu(\bta) , \left(\frac12\,\I\,+\,\K \right)\bva \right\rangle_\Gamma  
\end{array}
\end{equation}
and
\begin{equation}\label{b-N-complete-N-0}
\mathbf b_N((\bta,\bps),(\bv,\bet))\,:=\,
\int_{\Omega^-} \bv \cdot \bdiv\,\bta\,+\,\int_{\Omega^-} \bet : \bta \,,
\end{equation}
and $\F_N : \X_N \to \R$ and $\G_N : \Y_N \to \R$ are the bounded linear functionals given by
\begin{equation}\label{F-N-complete-N-0}
\F_N(\bta,\bps)\,:=\,0 \,,
\qquad
\mbox{and}
\qquad
\G_N(\bv,\bet,\bxi)\,:=\,-\,\int_{\Omega^-}\f\cdot\bv \,.
\end{equation}

\medskip
Concerning the solvability analysis of \eqref{dual-mixed-formulation-Neumann-complete-0}, 
we first observe that the continuous inf-sup condition for $\mathbf b_N$ was already
proved by Lemma \ref{inf-sup-b}. In addition, it is clear that
the kernel of $\mathbf b_N$ is given by
\[
\V_N\,:=\,\Big\{ \,(\bta,\bps)\in \X_N: \quad \bta = \bta^{\tt t} \qan \bdiv \bta \,=\,\0\qin
\Omega^-\,\Big\}\,,
\]
and that $\mathbf a_N$ satisfies the same positiveness property from Lemma \ref{quasi-coerciveness-a-N},
that is
\begin{equation}\label{quasi-coerciveness-a-N-new}
\a_N((\bta,\bps),(\bta,\bps))\,\ge\,\widetilde\alpha\,\Big\{\,
\|\bta\|^2_{\bdiv;\Omega^-}\,+\,\|\bps-\boldsymbol\pi_{_{RM}}\bps\|^2_{1/2,\Gamma}\,\Big\}
\qquad\forall\,(\bta,\bps)\,\in\,\V_N\,.
\end{equation}
Moreover, basically the same proof of Lemma \ref{kernel-Neumann} shows that the
set of solutions of the homogeneous version of \eqref{dual-mixed-formulation-Neumann-complete-0}
is given by
\[
\Big\{\,\big((\bsi,\bva),(\u,\bch)\big)\,:=\,
\big((\0,\z),(\0,\0)\big): \qquad \z\,\in\,\mathbf{RM}(\Gamma)\,\Big\}\,.
\]
Consequently, following a similar analysis to the one from the previous section,
and in order to guarantee unique solvability of the resulting problem,
\eqref{dual-mixed-formulation-Neumann-complete-0} is reformulated as: 
Find $((\bsi,\bva),(\u,\bch))\,\in\,{\widetilde\X_N}\times{\Y_N}$ such that
\begin{equation}\label{dual-mixed-formulation-Neumann-complete-0-new}
\begin{array}{rcl}
\mathbf a_N((\bsi,\bva),(\bta,\bps))\,+\,\mathbf b_N((\bta,\bps),(\u,\bch)) &=& \F_N(\bta,\bps)
\qquad\forall\,(\bta,\bps)\in \widetilde\X_N\,,\\[2ex]
\mathbf b_N((\bsi,\bva),(\bv,\bet)) &=& \G_N(\bv,\bet)
\qquad\forall\,(\bv,\bet)\in \Y_N\,,
\end{array}
\end{equation}
where 
\[
\widetilde{\X}_N\,:=\,\mathbb H_0(\bdiv;\Omega^-)\times\mathbf H^{1/2}_0(\Gamma)\,.
\]
Then, the kernel of $\mathbf b_N : \widetilde\X_N \times \Y_N \to \R$ becomes now
\[
\widetilde \V_N\,:=\,\Big\{ \,(\bta,\bps)\in \widetilde
\X_N: \quad \bta = \bta^{\tt t} \qan \bdiv \bta \,=\,\0\qin
\Omega^-\,\Big\}\,,
\]
whence \eqref{quasi-coerciveness-a-N-new} yields the strong coerciveness of $\mathbf a_N$ on 
$\widetilde \V_N$. In turn, since $\mathbf b_N$ does not depend on the component
$\bps\in \mathbf H^{1/2}_0(\Gamma)$, it is also clear from Lemma \ref{inf-sup-b} 
that $\mathbf b_N$ satisfies the continuous inf-sup condition on $\widetilde \X_N \times \Y_N$ as well.
These remarks and \cite[Theorem 1.1, Chapter II]{bf-1991}
imply the well-posedness of \eqref{dual-mixed-formulation-Neumann-complete-0-new}.
%
%
\begin{theorem}\label{CH-hNbc}
Given $\mathbf f\,\in\,\mathbf L^2(\Omega^-)$, there
exists a unique $((\bsi,\bva),(\u,\bch))\,\in\,{\widetilde\X_N}\times{\Y_N}$ solution to
\eqref{dual-mixed-formulation-Neumann-complete-0-new}. In addition, there exists $C > 0$ such that
\[
\|((\bsi,\bva),(\u,\bch))\|_{\X_N \times \Y_N}\,\le\,C\,\|\mathbf f\|_{0,\Omega^-}\,.
\]
\end{theorem}
\section{Galerkin schemes of the coupled formulations}\label{section6}
In this section we study the well-posedness of the Galerkin schemes associated with
each one of the coupled variational formulations analyzed in Section \ref{section5}.
\subsection{J \& N coupling with homogeneous Neumann on $\Gamma_0$}\label{section41}
We first let $\mathbb H^\bsi_{h,0}$, $\mathbf H^\bva_{\tilde h}$, $\mathbf L^\u_h$, and $\mathbb
L^\bch_h$ be finite dimensional subspaces of $\mathbb H_0(\bdiv;\Omega^-)$, $\mathbf
H^{1/2}(\Gamma)$, $\mathbf L^2(\Omega^-)$, and $\mathbb L^2_{\tt skew}(\Omega^-)$,
respectively, and define
\begin{equation}\label{eq-61-0}
\mathbf H^\bva_{\tilde h,0}\,:=\,\Big\{ \bps\in \mathbf H^\bva_{\tilde h}: \quad 
\langle \mathbf r, \bps\rangle_{1/2,\Gamma}
\,=\,0 \quad \forall\,\mathbf r \in \mathbf{RM}(\Gamma)\Big\}\,,
\end{equation}
which is clearly a subspace of $\mathbf H^{1/2}_0(\Gamma)$.
Note that, because of reasons that will become clear below, we take a different
meshsize for defining $\mathbf H^\bva_{\tilde h}$ (and hence $\mathbf H^\bva_{\tilde h,0}$). 
Then we introduce the product spaces
\[
\widetilde \X_{N,h}\,:=\, \mathbb H^\bsi_{h,0} \,\times\,\mathbf H^\bva_{\tilde h,0}
\qan \Y_{N,h}\,:=\,\mathbf L^\u_h \,\times\, \mathbb L^\bch_h\,,
\]
and define the Galerkin scheme associated with \eqref{dual-mixed-formulation-complete-new} 
as: Find $((\bsi_h,\bva_{\tilde h}),(\u_h,\bch_h))\,\in\,
{\widetilde{\X}_{N,h}}\times{\Y_{N,h}}$ such that
\begin{equation}\label{dual-mixed-formulation-complete-new-h}
\begin{array}{rcl}
\mathbf a_N((\bsi_h,\bva_{\tilde h}),(\bta,\bps))\,+\,\mathbf b_N((\bta,\bps),(\u_h,\bch_h)) &=& \F_N(\bta,\bps)
\qquad\forall\,(\bta,\bps)\,\in\, {\widetilde{\X}_{N,h}}\,,\\[2ex]
\mathbf b_N((\bsi_h,\bva_{\tilde h}),(\bv,\bet)) &=& \G_N(\bv,\bet)
\qquad\forall\,(\bv,\bet)\,\in\, {\Y_{N,h}}\,,
\end{array}
\end{equation}
where $\a_N$, $\b_N$, $\F_N$, and $\G_N$ are the bilinear forms and functionals defined in 
Section \ref{section331}.

\medskip
In order to prove the unique solvability, stability, and convergence 
of \eqref{dual-mixed-formulation-complete-new-h}, we have in mind the discrete
Babu\v ska-Brezzi theory and consider in what follows the following assumptions:

\medskip\noindent
{\bf (H.1)} the bilinear form $\mathbf b_N$ satisfies the discrete inf-sup condition
uniformly on $\widetilde \X_{N,h} \,\times \Y_{N,h}$, that is there exists $\tilde \beta > 0$,
independent of $h$ and $\tilde h$, such that 
\begin{equation}\label{discrete-inf-sup-b-N}
\sup_{(\bta_h,\bps_{\tilde h})\in\widetilde\X_{N,h} \backslash \{\0\}}
\frac{\mathbf b_N((\bta_h,\bps_{\tilde h}),(\bv,\bet))}{\|(\bta_h,\bps_{\tilde h})\|_{\X_N}}\,\ge\,
\tilde\beta\,\|(\bv,\bet)\|_{\Y_N}
\qquad\forall \,(\bv,\bet)\in \Y_{N,h}\,.
\end{equation}

\medskip\noindent
{\bf (H.2)} $\bdiv\,\mathbb H^\bsi_{h,0}\,\subseteq\,\mathbf L^\u_h$.

\medskip\noindent
{\bf (H.3)} there exists $\epsilon \in (0,1/2)$ such that 
$\mathbf H^\bva_{\tilde h,0}\,\subseteq\,\mathbf H^{1/2+\epsilon}(\Gamma)$ 
for each $\tilde h > 0$.

\medskip\noindent
{\bf (H.4)} the finite element subspace $\mathbf H^\bva_{\tilde h,0}$ satisfies the inverse
inequality, that is there exists $C > 0$, independent of $\tilde h$, such that
\begin{equation}\label{inverse-tilde-h}
\|\bva_{\tilde h}\|_{1/2+\delta,\Gamma}\,\le\,C\,\tilde h^{-\delta}\,
\|\bva_{\tilde h}\|_{1/2,\Gamma} \qquad\forall\,\bva_{\tilde h}\in \mathbf H^\bva_{\tilde h,0}\,,
\quad \forall\,\delta\,\in\,[0,\epsilon]\,.
\end{equation}

\medskip\noindent
{\bf (H.5)} the orthogonal projector $\Pi^\bch_h : \mathbb L^2_{\tt skew}(\Omega^-)
\to \mathbb L^\bch_h$ satisfies the approximation property
\begin{equation}\label{approximation-L-bch}
\|\bet - \Pi^\bch_h(\bet)\|_{0,\Omega^-}\,\le\,C\,h^\delta\,\|\bet\|_{\delta,\Omega^-}
\qquad\forall\,\bet\,\in\,\mathbb L^2_{\tt skew}(\Omega^-) \,\cap\,\mathbb H^\delta(\Omega^-)\,,
\quad \forall\,\delta\in [0,1]\,. 
\end{equation}

\medskip
We notice that, being the bilinear form $\mathbf b_N$ independent of the
$\bps$-component, its discrete inf-sup condition (cf. \eqref{discrete-inf-sup-b-N} in
{\bf (H.1)}) involves only the subspaces $\mathbb H^\bsi_{h,0}$, $\mathbf L^\u_h$, and 
$\mathbb L^\bch_h$. In addition, since the discrete kernel of $\mathbf b_N$ becomes
\begin{equation}\label{tilde-V-N-h}
\begin{array}{c}
\disp
\widetilde\V_{N,h}\,:=\,\Bigg\{\,(\bta_h,\bps_{\tilde h})\in\widetilde{\X}_{N,h}\,:=\,
\mathbb H^\bsi_{h,0} \times \mathbf H^\bva_{\tilde h,0}: \quad
\int_{\Omega^-} \bv_h \cdot\bdiv \bta_h\,=\,0 \quad \forall\,\bv_h\,\in\, \mathbf L^\u_h\\[2ex]
\disp\qan  \int_{\Omega^-} \bta_h : \bet_h\,=\,0 \quad \forall\,\bet_h\,\in\,\mathbb L^\bch_h    \,\Bigg\}\,,
\end{array}
\end{equation}
it is straightforward to see that hypothesis {\bf (H.2)} yields
\begin{equation}\label{free-divergence}
\bdiv \bta_h\,=\,\0 \qin \Omega^- \qquad \forall \, (\bta_h,\bps_{\tilde h}) \,\in\, \widetilde\V_{N,h}\,.
\end{equation}

\medskip
The statement \eqref{free-divergence} and hypotheses {\bf (H.3)} - {\bf (H.5)} are utilized
next to prove that $\mathbf a_N$ is uniformly strongly coercive on $\widetilde\V_{N,h}$.
The fact that $\D : \mathbf H^{1/2+\delta}(\Gamma) \to \mathbf H^{1+\delta}(\Omega^-)$ is a bounded 
linear operator for each $\delta \in [0,1/2)$, which
follows from similar arguments to those mentioned in the proof of Lemma  \ref{properties-1},
and which certainly extends the corresponding result for $\D$, will also be employed.
\begin{lemma}\label{a-N-coercive-discrete}
There exist $\,\,\tilde\alpha,\,c_0 > 0$, independent of $h$ and $\tilde h$, such that whenever 
$h\,\le\,c_0\,\tilde h$, there holds
\begin{equation}\label{eq-a-N-coercive-discrete}
\mathbf a_N\big((\bta_h,\bps_{\tilde h}),(\bta_h,\bps_{\tilde h})\big)\,\,\ge\,\,\tilde\alpha\,
\|(\bta_h,\bps_{\tilde h})\|^2_{\X_N} \qquad \forall\,(\bta_h,\bps_{\tilde h})\,\in\,
\widetilde\V_{N,h}\,.
\end{equation}
\end{lemma}
\begin{proof}
Given $(\bta_h,\bps_{\tilde h})\,\in\, \widetilde\V_{N,h}$, we first proceed as in the proof of
Lemma \ref{quasi-coerciveness-a} and obtain
\[
\mathbf a_N\big((\bta_h,\bps_{\tilde h}),(\bta_h,\bps_{\tilde h})\big)\,=\,
\|\bta^{\tt d}_h\|^2_{0,\Omega^-}\,+\,\langle \W\bps_{\tilde h},\bps_{\tilde h}\rangle_\Gamma
\,+\,\int_{\Omega^-} \nabla\big(\D\,\bps_{\tilde h}\big) : \bta^{\tt d}_h\,.
\]
Note that, while $\bta_h$ is divergence free (according to \eqref{free-divergence}), its lack
of strong symmetry stops us of replacing 
$\,\disp\int_{\Omega^-} \nabla\big(\D\,\bps_{\tilde h}\big) : \bta^{\tt d}_h$
 \,\,exactly \,by\,\, $\disp\int_{\Omega^-} \mathbf e\big(\D\,\bps_{\tilde h}\big) : \bta^{\tt d}_h$, 
as we did in that proof. However, what we can certainly do in the present discrete case is to write
\[
\nabla\big(\D\,\bps_{\tilde h}\big)\,\,=\,\,\mathbf e\big(\D\,\bps_{\tilde h}\big)
\,\,+\,\,\bet\big(\D\,\bps_{\tilde h}\big)\,,
\]
where
\[
\bet\big(\D\,\bps_{\tilde h}\big)\,=\,\frac12\,\Big\{\,\nabla\big(\D\,\bps_{\tilde h}\big)
\,+\,\Big(\nabla\big(\D\,\bps_{\tilde h}\big)\Big)^{\tt t}\,\Big\}\,.
\]
In this way, we have that
\[
\mathbf a_N\big((\bta_h,\bps_{\tilde h}),(\bta_h,\bps_{\tilde h})\big)\,=\,
\|\bta^{\tt d}_h\|^2_{0,\Omega^-}\,+\,\langle \W\bps_{\tilde h},\bps_{\tilde h}\rangle_\Gamma
\,+\,\int_{\Omega^-} \mathbf e\big(\D\,\bps_{\tilde h}\big) : \bta^{\tt d}_h\,+\,
\int_{\Omega^-} \bet \big(\D\,\bps_{\tilde h}\big) : \bta^{\tt d}_h\,,
\]
which, following the last part of the proof of Lemma \ref{quasi-coerciveness-a}, yields
\begin{equation}\label{eq-1-a-Ncd}
\mathbf a_N\big((\bta_h,\bps_{\tilde h}),(\bta_h,\bps_{\tilde h})\big)\,\ge\,
\frac12\,\Big\{\|\bta^{\tt d}_h\|^2_{0,\Omega^-}\,+\,\langle \W\bps_{\tilde h},\bps_{\tilde h}\rangle_\Gamma\,\Big\}
\,\,-\,\,\Big|\int_{\Omega^-} \bet \big(\D\,\bps_{\tilde h}\big) : \bta^{\tt d}_h\,\Big|\,.
\end{equation}
Next, using from \eqref{tilde-V-N-h} 
that $\disp \int_{\Omega^-} \bta^{\tt d}_h : \bet_h\,=\,\int_{\Omega^-} \bta_h : \bet_h\,=\,0 
\quad \forall\,\bet_h\,\in\,\mathbb L^\bch_h$, we find that
\[
\int_{\Omega^-} \bet \big(\D\,\bps_{\tilde h}\big) : \bta^{\tt d}_h\,=\,
\int_{\Omega^-} \Big\{\, \bet \big(\D\,\bps_{\tilde h}\big)\,-\,
\Pi^\bch_h\big(\bet \big(\D\,\bps_{\tilde h}\big)\big) \,\Big\} : \bta^{\tt d}_h\,,
\]
from which, applying the approximation property of $\mathbb L^\bch_h$ (cf. \eqref{approximation-L-bch}
in {\bf (H.5)}), {\bf (H.3)}, the boundedness of $\D : \mathbf H^{1/2+\epsilon}(\Gamma) 
\to \mathbf H^{1+\epsilon}(\Omega^-)$, and the inverse inequality satisfied by
$\mathbf H^\bva_{\tilde h,0}$ (cf. \eqref{inverse-tilde-h} in {\bf (H.4)}), 
we deduce that
\begin{equation}\label{eq-2-a-Ncd}
\begin{array}{l}
\disp
\Big|\int_{\Omega^-} \bet \big(\D\,\bps_{\tilde h}\big) : \bta^{\tt d}_h\,\Big|\,\le\,
\big\|\bet \big(\D\,\bps_{\tilde h}\big)\,-\,
\Pi^\bch_h\big(\bet \big(\D\,\bps_{\tilde h}\big)\big)\big\|_{0,\Omega^-}
\,\,\|\bta^{\tt d}_h\|_{0,\Omega^-}\\[2ex]
\disp\qquad\quad
\le \,\, C\,h^\epsilon\,\big\|\bet \big(\D\,\bps_{\tilde h}\big)\big\|_{\epsilon,\Omega^-}\,\,
\|\bta^{\tt d}_h\|_{0,\Omega^-}\,\le\,
C\,h^\epsilon\,\big\|\nabla \D\,\bps_{\tilde h}\big\|_{\epsilon,\Omega^-}\,\,
\|\bta^{\tt d}_h\|_{0,\Omega^-}\\[2ex]
\disp\qquad\quad
\le \,\,C\,h^\epsilon\,\big\|\D\,\bps_{\tilde h}\big\|_{1+\epsilon,\Omega^-}\,\,
\|\bta^{\tt d}_h\|_{0,\Omega^-}\,\le\,C\,h^\epsilon\,\big\|\bps_{\tilde h}\big\|_{1/2+\epsilon,\Gamma}\,\,
\|\bta^{\tt d}_h\|_{0,\Omega^-}\\[2ex]
\disp\qquad\quad
\le \,\, C\,\left\{\frac{h}{\tilde h}\right\}^\epsilon\,\big\|\bps_{\tilde h}\big\|_{1/2,\Gamma}\,\,
\|\bta^{\tt d}_h\|_{0,\Omega^-} \,\le\,C\,\left\{\frac{h}{\tilde h}\right\}^\epsilon\,
\left\{\,\frac12\,\big\|\bps_{\tilde h}\big\|^2_{1/2,\Gamma}\,+\,
\frac12\,\|\bta^{\tt d}_h\|^2_{0,\Omega^-}\,\right\}\,.
\end{array}
\end{equation}
Therefore, using that $\langle \W\,\bps,\bps\rangle_\Gamma 
\,\ge\,\alpha_2\,\|\bps\|^2_{1/2,\Gamma}\quad\forall\,\bps\,\in\,\mathbf H^{1/2}_0(\Gamma)$
(cf. \eqref{W-eliptico}), it follows from \eqref{eq-1-a-Ncd} and \eqref{eq-2-a-Ncd} that 
\[
\mathbf a_N\big((\bta_h,\bps_{\tilde h}),(\bta_h,\bps_{\tilde h})\big)\,\ge\,
\frac12\,\left\{1\,-\, C\,\left\{\frac{h}{\tilde h}\right\}^\epsilon \right\}\,\|\bta^{\tt d}_h\|^2_{0,\Omega^-}
\,+\,\frac12\,\left\{\alpha_2\,-\, C\,\left\{\frac{h}{\tilde h}\right\}^\epsilon \right\}\,
\|\bps_{\tilde h}\|^2_{1/2,\Gamma}\,,
\]
which yields the existence of a sufficiently small constant $c_0 > 0$ such that
for each $h \,\le\,c_0\,\tilde h$, $\mathbf a_N$ is uniformly strongly coercive
on $\widetilde\V_{N,h}$.

\end{proof}

\medskip
The well-posedness and convergence of \eqref{dual-mixed-formulation-complete-new-h} can now be established.
\begin{theorem}
Assume that the hypotheses {\bf (H.1)} up to {\bf (H.5)} are satisfied, and let $c_0$ be the
positive constant provided by Lemma {\rm \ref{a-N-coercive-discrete}}. Then, for
each $h \,\le\,c_0\,\tilde h$ there exists a unique $((\bsi_h,\bva_{\tilde h}),(\u_h,\bch_h))$
$\,\in\,{\widetilde{\X}_{N,h}}\times{\Y_{N,h}}$ solution of \eqref{dual-mixed-formulation-complete-new-h}.
In addition, there exist $C_1, \,C_2 \,>\,0$, independent of $h$ and $\tilde h$, such that
\[
\|((\bsi_h,\bva_{\tilde h}),(\u_h,\bch_h))\|_{\X_N \times \Y_N}\,\le\,C_1\,\|\mathbf f\|_{0,\Omega^-}\,,
\]
and
\begin{equation}\label{convergence}
\begin{array}{l}
\disp
\|((\bsi,\bva),(\u,\bch))\,-\,((\bsi_h,\bva_{\tilde h}),(\u_h,\bch_h))\|_{\X_N \times \Y_N}\\[2ex]
\disp\qquad
\,\le\,C_2\,\inf_{((\bta_h,\bps_{\tilde h}),(\bv_h,\bet_h))\in {\widetilde{\X}_{N,h}}\times{\Y_{N,h}}}
\|((\bsi,\bva),(\u,\bch))\,-\,((\bta_h,\bps_{\tilde h}),(\bv_h,\bet_h))\|_{\X_N \times \Y_N}\,.
\end{array}
\end{equation}

\end{theorem}

\subsection{C \& H coupling with non-homogeneous Dirichlet on $\Gamma_0$}\label{section62}
We now let $\mathbb H^\bsi_h$, $\mathbf H^\bva_h$, $\mathbf L^\u_h$, and $\mathbb
L^\bch_h$ be finite dimensional subspaces of $\mathbb H(\bdiv;\Omega^-)$, $\mathbf
H^{1/2}(\Gamma)$, $\mathbf L^2(\Omega^-)$, and $\mathbb L^2_{\tt skew}(\Omega^-)$,
respectively, and define
\begin{equation}\label{eq-62-0}
\mathbf H^\bva_{h,0}\,:=\,\Big\{ \bps\in \mathbf H^\bva_h: \quad 
\langle \mathbf r, \bps\rangle_{\Gamma}
\,=\,0 \quad \forall\,\mathbf r \in \mathbf{RM}(\Gamma)\Big\}\,=
\H^\bva_h\cap \mathbf H^{1/2}_0(\Gamma).
\end{equation}
Note that, differently from the previous section, in this case we do not need to take any different
meshsize for $\mathbf H^\bva_{h,0}$. Then we introduce the product spaces
\[
\widetilde \X_{D,h}\,:=\, \mathbb H^\bsi_h \,\times\,\mathbf H^\bva_{h,0}
\qan \Y_{D,h}\,:=\,\mathbf L^\u_h \,\times\, \mathbb L^\bch_h\,,
\]
and define the Galerkin scheme associated with \eqref{dual-mixed-formulation-Dirichlet-complete-new} as:
Find $((\bsi_h,\bva_h),(\u_h,\bch_h))\,\in\,
{\widetilde{\X}_{D,h}}\times{\Y_{D,h}}$ such that
\begin{equation}\label{dual-mixed-formulation-Dirichlet-complete-new-h}
\begin{array}{rcl}
\mathbf a_D((\bsi_h,\bva_),(\bta,\bps))\,+\,\mathbf b_D((\bta,\bps),(\u_h,\bch_h)) &=& \F_D(\bta,\bps)
\qquad\forall\,(\bta,\bps)\,\in\, {\widetilde{\X}_{D,h}}\,,\\[2ex]
\mathbf b_D((\bsi_h,\bva_h),(\bv,\bet)) &=& \G_D(\bv,\bet)
\qquad\forall\,(\bv,\bet)\,\in\, {\Y_{D,h}}\,,
\end{array}
\end{equation}
where $\a_D$, $\b_D$, $\F_D$, and $\G_D$ are the bilinear forms and functionals defined in 
Section \ref{section332}.

\medskip
Next, we apply again the discrete Babu\v ska-Brezzi theory to prove the unique solvability, stability, 
and convergence of \eqref{dual-mixed-formulation-Dirichlet-complete-new-h}. To this end,
in what follows we assume the following hypotheses:

\medskip\noindent
$\widetilde{\bf (H.1)}$ the bilinear form $\mathbf b_D$ satisfies the discrete inf-sup condition
uniformly on $\widetilde \X_{D,h} \,\times \Y_{D,h}$, that is there exists $\tilde \beta > 0$,
independent of $h$, such that 
\begin{equation}\label{discrete-inf-sup-b-D}
\sup_{(\bta_h,\bps_h)\in\widetilde\X_{D,h} \backslash \{\0\}}
\frac{\mathbf b_D((\bta_h,\bps_h),(\bv,\bet))}{\|(\bta_h,\bps_h)\|_{\X_D}}\,\ge\,
\tilde\beta\,\|(\bv,\bet)\|_{\Y_D}
\qquad\forall \,(\bv,\bet)\in \Y_{D,h}\,.
\end{equation}

\medskip\noindent
$\widetilde{\bf (H.2)}$ $\bdiv\,\mathbb H^\bsi_h\,\subseteq\,\mathbf L^\u_h$.

\medskip\noindent
$\widetilde{\bf (H.3)}$ there holds $\,P_0(\Omega^-)\,\I \,\subseteq\, \mathbb H^\bsi_h$ \,and\,
$\mathbf{RM}(\Gamma) \,\subseteq\, \mathbf H^\bva_h$.

\medskip\noindent
$\widetilde{\bf (H.4)}$ there exists $\,\widetilde\bps \in \mathbf H^{1/2}(\Gamma)$ such that
$\,\langle \bnu,\widetilde\bps\rangle_\Gamma \,>\,0$ \,and\, $\widetilde\bps\in \mathbf H^\bva_h$
\,for each $h > 0$.

\medskip
Similarly as in the previous section, the bilinear form $\mathbf b_D$ is independent of the $\bps$-component,
and hence its discrete inf-sup condition (cf. \eqref{discrete-inf-sup-b-D} in
$\widetilde{\bf (H.1)}$) involves only the subspaces $\mathbb H^\bsi_{h}$, $\mathbf L^\u_h$, and 
$\mathbb L^\bch_h$. In addition, hypothesis $\widetilde{\bf (H.2)}$ implies now 
that the discrete kernel of $\mathbf b_D$ reduces to
\begin{equation}\label{tilde-V-D-h}
\widetilde\V_{D,h} := \Bigg\{(\bta_h,\bps_h)\in\widetilde{\X}_{D,h}: \quad
\bdiv \bta_h = \0 \qin \Omega^-
\qan  \int_{\Omega^-} \bta_h : \bet_h = 0 \quad \forall\,\bet_h\,\in\,\mathbb L^\bch_h \Bigg\}\,.
\end{equation}

\medskip
As announced right before the statement of Lemma \ref{weak-coercive-a-D-direct}, we now
establish the discrete weak-coerciveness of $\a_D$.
\begin{lemma}\label{weak-coercive-a-D-direct-discrete}
There exists $\widehat\alpha > 0$, independent of $h$ but depending explicitly on $\widetilde\bps$, 
such that
\begin{equation}\label{eq-weak-coercive-a-D-direct-discrete}
\sup_{(\bta,\bps)\in\widetilde\V_{D,h} \backslash\{\0\}}
\frac{|\,\a_D((\bsi_h,\bva_h),(\bta,\bps))\,|}{\|(\bta,\bps)\|_{\X_D}}\,\ge\,
\widehat\alpha\,\|(\bsi_h,\bva_h)\|_{\X_D}
\qquad\forall\,(\bsi_h,\bva_h)\,\in\,\widetilde\V_{D,h}\,.
\end{equation}
\end{lemma}
\begin{proof}
This is a direct consequence of Lemma \ref{weak-coercive-a-D-direct} using  hypotheses $\widetilde{\bf (H.3)}$ and $\widetilde{\bf (H.4)}$. 
\end{proof}

\medskip
As a consequence of the foregoing analysis, we can provide now the well-posedness and convergence of 
\eqref{dual-mixed-formulation-Dirichlet-complete-new-h}.
\begin{theorem}\label{CH-coupling-NH-D}
Assume that the hypotheses $\widetilde{\bf (H.1)}$ up to $\widetilde{\bf (H.4)}$ are satisfied. Then, 
there exists a unique $((\bsi_h,\bva_h),(\u_h,\bch_h))$
$\,\in\,{\widetilde{\X}_{D,h}}\times{\Y_{D,h}}$ solution of \eqref{dual-mixed-formulation-Dirichlet-complete-new-h}.
In addition, there exist $C_1, \,C_2 \,>\,0$, independent of $h$, such that
\[
\|((\bsi_h,\bva_h),(\u_h,\bch_h))\|_{\X_D \times \Y_D}\,\le\,C_1\,
\Big\{ \|\g_{_D}\|_{1/2,\Gamma_0} \,+\, \|\mathbf f\|_{0,\Omega^-}\Big\}\,,
\]
and
\begin{equation}\label{convergence-CH-NH-D}
\begin{array}{l}
\disp
\|((\bsi,\bva),(\u,\bch))\,-\,((\bsi_h,\bva_h),(\u_h,\bch_h))\|_{\X_D \times \Y_D}\\[2ex]
\disp\qquad
\,\le\,C_2\,\inf_{((\bta_h,\bps_h),(\bv_h,\bet_h))\in {\widetilde{\X}_{D,h}}\times{\Y_{D,h}}}
\|((\bsi,\bva),(\u,\bch))\,-\,((\bta_h,\bps_h),(\bv_h,\bet_h))\|_{\X_D \times \Y_D}\,.
\end{array}
\end{equation}

\end{theorem}

\subsection{C \& H coupling with non-homogeneous Neumann on $\Gamma_0$}\label{section63}
We let $\mathbb H^\bsi_h$, $\mathbf H^\bva_h$, $\mathbf L^\u_h$, $\mathbb
L^\bch_h$, and $\mathbf H^\lambda_{\tilde h}$ be finite dimensional subspaces of $\mathbb H(\bdiv;\Omega^-)$, $\mathbf
H^{1/2}(\Gamma)$, $\mathbf L^2(\Omega^-)$, $\mathbb L^2_{\tt skew}(\Omega^-)$, and $\mathbf H^{1/2}(\Gamma_0)$,
respectively, and define, as in \eqref{eq-62-0}, the subspace 
\begin{equation}\label{eq-63-0}
\mathbf H^\bva_{h,0} \,:=\, \mathbf H^\bva_{h}\cap\mathbf H^{1/2}_0(\Gamma)\,.
\end{equation}
Note that, in order to be able to define specific discrete subspaces satisfying the assumptions to be 
specified below, we need to take a different meshsize for the finite element subspace of $\mathbf H^{1/2}(\Gamma_0)$. 
However, as in the previous section, this is not required for $\mathbf H^\bva_{h,0}$. Then we introduce the product spaces
\[
\widetilde \X_{N,h}\,:=\, \mathbb H^\bsi_h \,\times\,\mathbf H^\bva_{h,0}
\qan \Y_{N,h,\tilde h}\,:=\,\mathbf L^\u_h \,\times\, \mathbb L^\bch_h \times \mathbf H^\lambda_{\tilde h}\,,
\]
and define the Galerkin scheme associated with \eqref{dual-mixed-formulation-Neumann-complete-new}  as:
Find $((\bsi_h,\bva_h),(\u_h,\bch_h,\bla_{\tilde h}))\,\in\,
{\widetilde{\X}_{N,h}}\times{\Y_{N,h,\tilde h}}$ such that
\begin{equation}\label{dual-mixed-formulation-Neumann-complete-new-h}
\begin{array}{rcl}
\mathbf a_N((\bsi_h,\bva_),(\bta,\bps))\,+\,\mathbf b_N((\bta,\bps),(\u_h,\bch_h,\bla_{\tilde h})) &=& \F_N(\bta,\bps)
\qquad\forall\,(\bta,\bps)\,\in\, {\widetilde{\X}_{N,h}}\,,\\[2ex]
\mathbf b_N((\bsi_h,\bva_h),(\bv,\bet,\bxi)) &=& \G_N(\bv,\bet,\bxi)
\qquad\forall\,(\bv,\bet,\bxi)\,\in\, {\Y_{N,h,\tilde h}}\,,
\end{array}
\end{equation}
where $\a_N$, $\b_N$, $\F_N$, and $\G_N$ are those bilinear forms and functionals defined in 
Section \ref{section333}.

\medskip
Proceeding as before, in what follows we apply the discrete Babu\v ska-Brezzi theory to establish the well-posedness and 
convergence of \eqref{dual-mixed-formulation-Neumann-complete-new-h}. For this purpose,
we now assume the following hypotheses:

\medskip\noindent
$\widehat{\bf (H.1)}$ the bilinear form $\mathbf b_N$ satisfies the discrete inf-sup condition
uniformly on $\widetilde \X_{N,h} \,\times \Y_{N,h,\tilde h}$, that is there exists $\tilde \beta > 0$,
independent of $h$ and $\tilde h$, such that 
\begin{equation}\label{discrete-inf-sup-b-N-CH}
\sup_{(\bta_h,\bps_h)\in\widetilde\X_{N,h} \backslash \{\0\}}
\frac{\mathbf b_N((\bta_h,\bps_h),(\bv,\bet,\bxi))}{\|(\bta_h,\bps_h)\|_{\X_N}}\,\ge\,
\tilde\beta\,\|(\bv,\bet,\bxi)\|_{\Y_N}
\qquad\forall \,(\bv,\bet,\bxi)\in \Y_{N,h,\tilde h}\,.
\end{equation}

\medskip\noindent
$\widehat{\bf (H.2)}$ $\bdiv\,\mathbb H^\bsi_h\,\subseteq\,\mathbf L^\u_h$.

\medskip\noindent
$\widehat{\bf (H.3)}$ there holds $\,P_0(\Omega^-)\,\I \,\subseteq\, \mathbb H^\bsi_h$.

\medskip\noindent
$\widehat{\bf (H.4)}$ there exists $\,\widetilde\bxi \in \mathbf H^{1/2}(\Gamma_0)$ such that
$\,\langle \bnu,\widetilde\bxi\rangle_{\Gamma_0} \,>\,0$ \,and\, $\widetilde\bxi\in \mathbf H^\lambda_{\tilde h}$
\,for each $h > 0$.

\medskip
Note, as in both previous sections, that the bilinear form $\mathbf b_N$ is also independent of the $\bps$-component,
and hence its discrete inf-sup condition (cf. \eqref{discrete-inf-sup-b-N-CH} in
$\widehat{\bf (H.1)}$) involves only the subspaces $\mathbb H^\bsi_{h}$, $\mathbf L^\u_h$, and 
$\mathbb L^\bch_h$. In addition, hypothesis $\widehat{\bf (H.2)}$ implies now 
that the discrete kernel of $\mathbf b_N$ reduces to
\begin{equation}\label{tilde-V-N-h-CH}
\begin{array}{c}
\disp
\widetilde\V_{N,h} := \Bigg\{(\bta_h,\bps_h)\in\widetilde{\X}_{N,h}: \quad
\bdiv \bta_h = \0 \qin \Omega^-\,, \quad \int_{\Omega^-} \bta_h : \bet_h = 0 
\quad \forall\,\bet_h\,\in\,\mathbb L^\bch_h\,, \\[2ex]
\disp
\qan \qquad \langle \gamma^-_\bnu(\bta_h),\bxi\rangle_{\Gamma_0}\,=\,0 \quad \forall\,\bxi\in \mathbf H^\lambda_{\tilde h}
\Bigg\}\,.
\end{array}
\end{equation}
In turn, $\widehat{\bf (H.3)}$ guarantees that there holds the decomposition (cf. \eqref{tilde-hbdiv} - \eqref{decomp-hdiv-1})
\[
\mathbb H^\bsi_h\,=\, \widetilde{\mathbb H}^\bsi_h\,\oplus\,P_0(\Omega^-)\,\I\,,
\]
where
\[
\widetilde{\mathbb H}^\bsi_h \,:=\,
\Big\{\,\bta_h\,\in\,\mathbb H^\bsi_h:
\quad \int_{\Omega^-} \tr \bta_h \,=\, 0\,\Big\}\,.
\]
In other words, $\bta_h-\boldsymbol\pi_{_I}\bta_h \in \widetilde{\mathbb H}^\bsi_h$ for each $\bta_h \in \mathbb H_h^\bsi$. Hence, the discrete analogue of Lemma \ref{otra-cota-tau} is established now as a consequence of $\widehat{\bf (H.4)}$.
\begin{lemma}\label{otra-cota-tau-discreto}
Let $\mathbb H_{h,\tilde h}\,:=\,\Big\{\,\bta_h\in \mathbb H^\bsi_h: \quad \langle\gamma^-_\bnu(\bta_h),\bxi\rangle_{\Gamma_0}\,=\,0
\quad\forall\,\bxi \,\in\,\mathbf H^\lambda_{\tilde h}\,\Big\}$. Then there exists $\widetilde c_2 > 0$, independent of $h$ and $\tilde h$, such that
\begin{equation}\label{eq-otra-cota-tau-discreto}
\|\bta_h\|^2_{\bdiv;\Omega^-}\,\,\le\,\,\widetilde c_2\,\|\bta_{h}-\boldsymbol\pi_{_I}\bta_h\|^2_{\bdiv;\Omega^-}
\qquad\forall\,\bta_h \,\in\,\mathbb H_{h,\tilde h}\,.
\end{equation}
\end{lemma}
\begin{proof}
Let $\bta_h\in \mathbb H_{h,\tilde h}$ and write $\boldsymbol\pi_{_I}\bta_h=d_h \I$. Then
\[
0\,=\,\langle \gamma^-_\bnu(\bta_h),\widetilde\bxi\rangle_{\Gamma_0}\,=\,\langle \gamma^-_\bnu(\bta_{h}-\boldsymbol\pi_{_I}\bta_h),\widetilde\bxi\rangle_{\Gamma_0} \,+\, d_h\,\langle \bnu,\widetilde\bxi\rangle_{\Gamma_0}\,,
\]
which gives
\[
|d_h|\,\,\le\,\,C\,\frac{\|\widetilde\bxi\|_{1/2,\Gamma_0}}{|\,\langle \bnu,\widetilde\bxi\rangle_{\Gamma_0}\,|}
\,\,\|\bta_{h}-\boldsymbol\pi_{_I}\bta_h\|_{\bdiv;\Omega^-}\,.
\]
This inequality and the fact that $\,\|\bta_h\|^2_{\bdiv;\Omega^-}\,=\,\|\bta_h-\boldsymbol\pi_{_I}
\bta_h\|^2_{\bdiv;\Omega^-}
\,+\,2\,d^2_h\,|\Omega^-|\,$ imply \eqref{eq-otra-cota-tau-discreto}.
\end{proof}

\medskip
In this way, noting that $\widetilde\V_{N,h} \,\subseteq\, \mathbb H_{h,\tilde h} \times \mathbf H^\bva_{h,0}$,
and using now Lemma \ref{otra-cota-tau-discreto} (instead of Lemma \ref{otra-cota-tau}) together with
\eqref{eq-cota-tau} (cf. Lemma \ref{cota-tau}), \eqref{W-eliptico-equiv}, and \eqref{V-eliptico-equiv}, 
we conclude the strong coerciveness of $\a_N$ on $\widetilde\V_{N,h}$. Therefore, we summarize the foregoing 
analysis in the following main result.
\begin{theorem}\label{CH-coupling-NH-N}
Assume that the hypotheses $\widehat{\bf (H.1)}$ up to $\widehat{\bf (H.4)}$ are satisfied. Then, 
there exists a unique $((\bsi_h,\bva_h),(\u_h,\bch_h,\bla_{\tilde h}))$
$\,\in\,{\widetilde{\X}_{N,h}}\times{\Y_{N,h,\tilde h}}$ solution of \eqref{dual-mixed-formulation-Neumann-complete-new-h}.
In addition, there exist $C_1, \,C_2 \,>\,0$, independent of $h$, such that
\[
\|((\bsi_h,\bva_h),(\u_h,\bch_h,\bla_{\tilde h}))\|_{\X_N \times \Y_N}\,\le\,C_1\,
\Big\{ \|\g_{_N}\|_{-1/2,\Gamma_0} \,+\, \|\mathbf f\|_{0,\Omega^-}\Big\}\,,
\]
and
\begin{equation}\label{convergence-CH-NH-N}
\begin{array}{l}
\disp
\|((\bsi,\bva),(\u,\bch,\bla))\,-\,((\bsi_h,\bva_h),(\u_h,\bch_h,\bla_{\tilde h}))\|_{\X_N \times \Y_N}\\[2ex]
\disp\qquad
\,\le\,C_2\,\inf_{((\bta_h,\bps_h),(\bv_h,\bet_h,\bxi_{\tilde h}))\in {\widetilde{\X}_{N,h}}\times{\Y_{N,h,\tilde h}}}
\|((\bsi,\bva),(\u,\bch,\bla))\,-\,((\bta_h,\bps_h),(\bv_h,\bet_h,\bxi_{\tilde h}))\|_{\X_N \times \Y_N}\,.
\end{array}
\end{equation}
\end{theorem}

\subsection{C \& H coupling with homogeneous Neumann on $\Gamma_0$}\label{section64}
We now let $\mathbb H^\bsi_{h,0}$, $\mathbf H^\bva_h$, $\mathbf L^\u_h$, and $\mathbb
L^\bch_h$ be finite dimensional subspaces of $\mathbb H_0(\bdiv;\Omega^-)$, $\mathbf
H^{1/2}(\Gamma)$, $\mathbf L^2(\Omega^-)$, and $\mathbb L^2_{\tt skew}(\Omega^-)$,
respectively, and define, as in previous occasions,
\begin{equation}\label{eq-64-0}
\mathbf H^\bva_{h,0}=\mathbf H^\bva_{h}\cap\mathbf H^{1/2}_0(\Gamma)\,.
\end{equation}
Then we introduce the product spaces
\[
\widetilde \X_{N,h}\,:=\, \mathbb H^\bsi_{h,0}\,\times\,\mathbf H^\bva_{h,0}
\qan \Y_{N,h}\,:=\,\mathbf L^\u_h \,\times\, \mathbb L^\bch_h \,,
\]
and define the Galerkin scheme associated with \eqref{dual-mixed-formulation-Neumann-complete-0-new}  as:
Find $((\bsi_h,\bva_h),(\u,\bch))\,\in\,{\widetilde\X_{N,h}}\times{\Y_{N,h}}$ such that
\begin{equation}\label{dual-mixed-formulation-Neumann-complete-0-new-h}
\begin{array}{rcl}
\mathbf a_N((\bsi_h,\bva_h),(\bta,\bps))\,+\,\mathbf b_N((\bta,\bps),(\u_h,\bch_h)) &=& \F_N(\bta,\bps)
\qquad\forall\,(\bta,\bps)\in \widetilde\X_{N,h}\,,\\[2ex]
\mathbf b_N((\bsi_h,\bva_h),(\bv,\bet)) &=& \G_N(\bv,\bet)
\qquad\forall\,(\bv,\bet)\in \Y_{N,h}\,,
\end{array}
\end{equation}
where $\a_N$, $\b_N$, $\F_N$, and $\G_N$ are those bilinear forms and functionals defined in 
Section \ref{section334}.

\medskip
We assume the following hypotheses:

\medskip\noindent
$\overline{\bf (H.1)}$ the bilinear form $\mathbf b_N$ satisfies the discrete inf-sup condition
uniformly on $\widetilde \X_{N,h} \,\times \Y_{N,h}$, that is there exists $\tilde \beta > 0$,
independent of $h$, such that 
\begin{equation}\label{discrete-inf-sup-b-N-0}
\sup_{(\bta_h,\bps_h)\in\widetilde\X_{N,h} \backslash \{\0\}}
\frac{\mathbf b_N((\bta_h,\bps_h),(\bv,\bet))}{\|(\bta_h,\bps_h)\|_{\X_N}}\,\ge\,
\tilde\beta\,\|(\bv,\bet)\|_{\Y_N}
\qquad\forall \,(\bv,\bet)\in \Y_{N,h}\,.
\end{equation}

\medskip\noindent
$\overline{\bf (H.2)}$ $\bdiv\,\mathbb H^\bsi_{h,0}\,\subseteq\,\mathbf L^\u_h$.

\medskip\noindent
$\overline{\bf (H.3)}$ there holds $\mathbf{RM}(\Gamma) \,\subseteq\, \mathbf H^\bva_h$.

\medskip
As in all the previous subsections, we note once again that the bilinear form $\b_N$
does not depend on the $\bps$-component, and hence the eventual verification of its
discrete inf-sup condition (cf. \eqref{discrete-inf-sup-b-N-0} in $\overline{\bf (H.1)}$) is
determined only by the subspaces $\mathbb H^\bsi_{h,0}$, $\mathbf L^\u_h$, and $L^\bch_h$.
In turn, it is also quite obvious from $\overline{\bf (H.2)}$ that the discrete kernel of $\mathbf b_N$ 
reduces to
\begin{equation}\label{tilde-V-N-h-CH-0}
\widetilde\V_{N,h} := \Bigg\{(\bta_h,\bps_h)\in\widetilde{\X}_{N,h}: \quad
\bdiv \bta_h = \0 \qin \Omega^-\,, \quad \int_{\Omega^-} \bta_h : \bet_h = 0 
\quad \forall\,\bet_h\,\in\,\mathbb L^\bch_h \Bigg\}\,.
\end{equation}

\medskip
Furthermore, it is clear that $\overline{\bf (H.3)}$ yields the decomposition
$\mathbf H^\bva_h \,=\,\mathbf H^\bva_{h,0} \,\oplus\, \mathbf{RM}(\Gamma)$, and
$\bps_h-\boldsymbol\pi_{_{RM}}\bps_h\in \H_{h,0}^\bva$ for all $\bps_h \in \H_h^\bva$.
Consequently,
\eqref{tilde-V-N-h-CH-0} and the application of \eqref{quasi-coerciveness-a-N-new}
to the above discrete context confirm that the bilinear form $\a_N$ is strongly
coercive on $\widetilde\V_{N,h}$.

\medskip
The well-posedness and convergence of \eqref{dual-mixed-formulation-Neumann-complete-0-new-h} 
is established next as a straightforward consequence of the foregoing analysis.
\begin{theorem}\label{CH-coupling-NH-N-0-H}
Assume that the hypotheses $\overline{\bf (H.1)}$ up to $\overline{\bf (H.3)}$ are satisfied. Then, 
there exists a unique $((\bsi_h,\bva_h),(\u_h,\bch_h))$
$\,\in\,{\widetilde{\X}_{N,h}}\times{\Y_{N,h}}$ solution of \eqref{dual-mixed-formulation-Neumann-complete-0-new-h}.
In addition, there exist $C_1, \,C_2 \,>\,0$, independent of $h$, such that
\[
\|((\bsi_h,\bva_h),(\u_h,\bch_h))\|_{\X_N \times \Y_N}\,\le\,C_1\, \|\mathbf f\|_{0,\Omega^-} \,,
\]
and
\begin{equation}\label{convergence-CH-NH-N-0}
\begin{array}{l}
\disp
\|((\bsi,\bva),(\u,\bch))\,-\,((\bsi_h,\bva_h),(\u_h,\bch_h))\|_{\X_N \times \Y_N}\\[2ex]
\disp\qquad
\,\le\,C_2\,\inf_{((\bta_h,\bps_h),(\bv_h,\bet_h))\in {\widetilde{\X}_{N,h}}\times{\Y_{N,h}}}
\|((\bsi,\bva),(\u,\bch))\,-\,((\bta_h,\bps_h),(\bv_h,\bet_h))\|_{\X_N \times \Y_N}\,.
\end{array}
\end{equation}
\end{theorem}

\bigskip
We end this paper by remarking that specific subspaces satisfying the hypotheses described in  each
one of the subsections of the present section can be easily found in the existing literature. Indeed, we first observe
that the expressions defining the bilinear forms $\b_N$, $\b_D$, and $\b_N$ involved in ${\bf (H.1)}$, 
$\widetilde{\bf (H.1)}$, and $\overline{\bf (H.1)}$, respectively, are the same as the one arising from the
mixed formulation of the linear elasticity problem in $\Omega^-$ with Dirichlet or homogeneous Neumann boundary 
conditions on $\partial\Omega^-$. Hence, any triple of finite element subspaces 
$(\mathbb H^\bsi_h,\mathbf L^\u_h,\mathbb L^\bch_h)$ 
(or $(\mathbb H^\bsi_{h,0},\mathbf L^\u_h,\mathbb L^\bch_h)$) yielding the discrete stability of that problem 
will satisfy our aforementioned hypotheses. In particular,
besides the classical PEERS (cf. \cite{abd-JJAM-1984}), we can consider for any integer $k \ge 1$ the PEERS$_k$ and the 
BDMS$_k$ elements, whose definitions and corresponding proofs of stability are provided in \cite{lv-NM-2004}.
In addition, new stable mixed finite element methods for 3D linear elasticity with a weak symmetry condition for the
stresses have been constructed recently in \cite{afw-IMA-2005} and \cite{afw-math-comp-2007} by
using the finite element exterior calculus. This is a quite abstract framework involving several sophisticated mathematical 
tools (see also \cite{afw-acta-2006} and \cite{aw-numerische-2002} for further details), which has been simplified
in some particular cases by employing more elementary and classical techniques (see, e.g. \cite{bbf-CPAA-2009}).
The resulting Arnold-Falk-Winther (AFW) element with the lowest polynomial degrees, which is referred to as of order $1$,
and which certainly constitutes another feasible choice verifying ${\bf (H.1)}$, $\widetilde{\bf (H.1)}$, and $\overline{\bf (H.1)}$,
consists of piecewise linear approximations for the stress $\bsi$ and piecewise constants functions for both the 
velocity $\u$ and vorticity $\bch$ unknowns.

\medskip
In turn, the bilinear form $\b_N$ involved in $\widehat{\bf (H.1)}$ corresponds also to the one arising from the
mixed formulation of the linear elasticity problem in $\Omega^-$, but now with Dirichlet boundary 
condition on $\Gamma$ and non-homogeneous Neumann boundary condition on $\Gamma_0$.
As stated at the beginning of Section \ref{section63}, we recall here that a different meshsize $\tilde h$
is assumed to define the finite element subspace $\mathbf H^\bla_{\tilde h}$ of $\mathbf H^{1/2}(\Gamma_0)$.
In other words, we consider a partition $\Gamma_{0,\tilde h}$ of $\Gamma_0$ that is independent of the
partition $\Gamma_{0,h}$ of $\Gamma_0$ inherited from a regular triangulation $\mathcal T_h$ of $\Omega^-$.
In this case, given an integer $k \ge 0$, and denoting the classical PEERS from \cite{abd-JJAM-1984}
by PEERS$_0$, we first take either PEERS$_k$ or BDMS$_k$ (when $k \ge 1$)
to define the triple $(\mathbb H^\bsi_h,\mathbf L^\u_h,\mathbb L^\bch_h)$, and 
then let $\mathbf H^\bla_{\tilde h}$ be the space of continuous piecewise polynomials on $\Gamma_{0,\tilde h}$ of 
degree $\le k+1$. Similarly, one could also take the above described AFW element of order $1$
and set $\mathbf H^\bla_{\tilde h}$ as the space of continuous piecewise polynomials on $\Gamma_{0,\tilde h}$ 
of degree $\le 1$. Hence, proceeding analogously as in the proofs of \cite[Lemmata 4.1, 4.2, and 4.3]{gmm-CMAME-2008},
one can easily show that for each one of the foregoing choices there exist $C_0, \,\tilde\beta  > 0$, independent of $h$ and
$\tilde h$, such that whenever $h \,\le\, C_0 \, \tilde h$, the hypothesis $\widehat{\bf (H.1)}$ is satisfied with the
constant $\tilde\beta$. We can also refer to \cite[Lemma 7.5]{ghm-SINUM-2010} for a closely related argument.

\medskip
Furthermore, it is easy to see that the above described pairs of finite element subspaces 
$\big(\mathbb H^\bsi_h,\mathbf L^\u_h\big)$ (or $\big(\mathbb H^\bsi_{h,0},\mathbf L^\u_h\big)$) 
also verify the respective hypotheses ${\bf (H.2)}$, $\widetilde{\bf (H.2)}$, $\widehat{\bf (H.2)}$, and 
$\overline{\bf (H.2)}$, which imply that the first components of the discrete kernels of $\b_N$ and $\b_D$
become all free divergent. In addition, it is clear that the corresponding subspaces $\mathbf L^\bch_h$ satisfy the 
approximation property required by ${\bf (H.5)}$. Moreover, it is quite straightforward to notice that the
first condition in $\widetilde{\bf (H.3)}$, and $\widehat{\bf (H.3)}$, are both satisfied by any of the previously 
mentioned choices of $\mathbb H^\bsi_h$.

\medskip
Next, with respect to the boundary element subspaces of $\mathbf H^{1/2}(\Gamma)$ and $\mathbf H^{1/2}_0(\Gamma)$, 
we first let $\Gamma_h$ and $\Gamma_{\tilde h}$ be independent partitions of $\Gamma$, with $\Gamma_h$ being 
the one inherited from a regular triangulation $\mathcal T_h$ of $\Omega^-$. Hence, in order to satisfy the regularity assumption
${\bf (H.3)}$ and the inverse inequality ${\bf (H.4)}$, it suffices to consider an integer $k \ge 0$,
set $\mathbf H^\bva_{\tilde h}$ as the space of continuous piecewise polynomials on $\Gamma_{\tilde h}$ 
of degree $\le k+1$, and define $\mathbf H^\bva_{\tilde h,0}$ as the intersection of $\mathbf H^{1/2}_0(\Gamma)$ 
with $\mathbf H^\bva_{\tilde h}$. Similarly, defining $\mathbf H^\bva_h$ as the space of continuous piecewise 
polynomials on $\Gamma_h$ of degree $\le k+1$, we find that the second condition in $\widetilde{\bf (H.3)}$, and
$\overline{\bf (H.3)}$, follow straightforwardly from the fact that $\mathbf{RM}(\Gamma)$ is certainly 
contained in the space of continuous piecewise polynomials on $\Gamma_h$ of degree $\le 1$.
On the other hand, concerning $\widetilde{\bf (H.4)}$ and $\widehat{\bf (H.4)}$, we just remark that the existence
of $\widetilde\bps\in \mathbf H^{1/2}(\Gamma)$ and $\widetilde\bxi\in \mathbf H^{1/2}(\Gamma_0)$
satisfying the required conditions in those hypotheses, follows by simply adapting the procedure suggested
in \cite[paragraph right before Lemma 8]{ggm-NM-2014} to the present 3D case (see also \cite[Section 3.2]{gos-MC-2011}).
To this end, we just need to define $\mathbf H^\bva_h$ (resp. $\mathbf H^\bla_{\tilde h}$) as 
any subspace of $\mathbf H^{1/2}(\Gamma)$ (resp. $\mathbf H^{1/2}(\Gamma_0)$) containing the space of continuous 
piecewise polynomials on $\Gamma_h$ (resp. $\Gamma_{0,h}$) of degree $\le 1$.
 
\medskip
Finally, while the definitions of $\mathbf H^\bva_{\tilde h,0}$ and $\mathbf H^\bva_{h,0}$ (cf. \eqref{eq-61-0}, 
\eqref{eq-62-0}, \eqref{eq-63-0}, and \eqref{eq-64-0}) are theoretically correct, we remark that for the sake of 
the computational implementation of the corresponding Galerkin schemes, it will be better off to introduce Lagrange 
multipliers handling the orthogonality conditions defining these boundary element subspaces. We omit further details
and leave this issue and other related matters, including numerical essays and the analysis of the
associated experimental rates of convergence, for a separate work.

\end{document}